\documentclass[10pt,a4paper,final]{article}
\usepackage[utf8]{inputenc}
\usepackage[lmargin=2.5cm,tmargin=2cm,bmargin=2cm,rmargin=2.5cm]{geometry}
\usepackage{amsfonts}
\usepackage{amssymb}
\usepackage{fix-cm}
\usepackage{amscd,amssymb,stmaryrd}
\usepackage{amsmath}
\usepackage{graphicx}
\usepackage{subfigure}
\usepackage{amscd,amstext}
\usepackage{hyperref}
\usepackage{color}
\usepackage{caption}
\usepackage{cite}
\usepackage{amsthm}
\usepackage{mathtools}
\usepackage{bigints}
\usepackage{amsrefs}
\usepackage{showlabels} 
\usepackage{enumitem}
\usepackage{mathrsfs}
\usepackage{empheq}

\usepackage{nameref}
\usepackage{esint}

\numberwithin{equation}{section}

\newcommand{\ds}{\displaystyle}

\newtheorem{theorem}{Theorem}[section]
\newtheorem{lemma}[theorem]{Lemma}

\newtheorem{definition}[theorem]{Definition}
\newtheorem{prop}[theorem]{Proposition}

\newtheorem*{theorem*}{Theorem}
\newtheorem*{lemma*}{Lemma}
\newtheorem*{conj*}{Conjecture}
\newtheorem*{corollary*}{Corollary}
\newtheorem*{proposition*}{Proposition}
\newtheorem{remark}[theorem]{Remark}
\newcommand{\rom}[1]{\uppercase\expandafter{\romannumeral #1\relax}}

\newcommand{\Z}{\mathbb{Z}}
\newcommand{\R}{\mathbb{R}}

\newcommand{\N}{\mathbb{N}}

\newcommand{\cH}{\mathcal{H}}

\newcommand{\Ga}{\alpha}
\newcommand{\Gb}{\beta}

\newcommand{\Ge}{\varepsilon}

\newcommand{\Gth}{\theta}

\newtheorem*{eg}{Example}

\DeclareMathOperator*{\Lip}{Lip}

\makeatletter
\newcommand{\labitem}[2]{%
\def\@itemlabel{\textbf{#1}}
\item
\def\@currentlabel{#1}\label{#2}}
\makeatother

\begin{document}

\begin{center}
	{\Large  On periodic solutions of the Benjamin-Bona-Mahony-Burgers equation}\\\vspace{0.25in} 
    Chun-Ho Lau \footnote[1]{Department of Mathematical Sciences, University of Cincinnati, Cincinnati, OH 45221, USA. Email: dchlau.math@gmail.com}
	Taige Wang \footnote[2]{Department of Mathematical Sciences, University of Cincinnati, Cincinnati, OH 45221, USA. Email: taige.wang@uc.edu}\ \ \ \
	   \vspace{0.06in}
\end{center}

\begin{abstract}
	In this paper, we would establish the existence and stability of periodic solutions to the Benjamin-Bona-Mahony-Burgers (BBM-Burgers) equation in $H^1_0([0, 1])$, whose medium interior is applied with time-periodic force $f(x, t)$ with period $\theta$. High regularity analysis has been conducted in Hilbert spaces $H^\ell, \ell>1$. We also consider periodic solution to same IBVP scenario of a pseudo-parabolic-regularized equation as an extension of the BBM-Burgers in $\cH^{\ell}, \ell=\{1, 2\}$. 
	
\end{abstract}
\section{Introduction}

\ In this paper, we first concern with the tempral-periodic solution of a BBM (Benjamin-Bona-Mahony) equation regularized by a viscous diffusion: 

\begin{equation}\label{BBM}
    u_t+u_x+uu_x-u_{xx}-u_{xxt}=f(x,t), 
\end{equation}

\noindent posed on a finite domain (a segment) $[0, 1]$ prescribed with homogeneous two-point boundary condition (Dirichlet boundary condition) and initial value: 

\begin{equation}\label{IBVP}
u(0,t)=u(1,t)=0,\,\,u(x, 0)=\phi(x).
\end{equation}

We also consider similar problem of solution to a pseudo-parabolic equation:

\begin{eqnarray}
\label{eqmain2}
u_t+u_x-u_{xx}-u_{xxt}+[F(u)]_x=[\Phi(u_x)]_x+(I-\partial_{xx})[G(u)]+f(x,t),  \quad (x, t)\in [0,1]\times [0, \infty)
\end{eqnarray} prescribed with same initial-boundary conditions (\ref{IBVP}). \\

Equation (\ref{BBM}) is BBM-Burgers equation. The original BBM equation was proposed in numerical fluid simulations in \cite{Peregrine} by Peregrine in 1960s and later systematically presented aligned with Korteweg-de Vries (KdV) equation in \cite{BBM} by Benjamin, Bona, and Mahony in 1970s. Both BBM and KdV models share same ground in proposing and are used to model the unidirectional propagation of small amplitude surface water waves formulated  by gravity; however, the KdV has longer identification in the past tracing back in 19th century by 
Boussinesq ({see e.g. his earliest papers on the model \cite{Bouss1, Bouss2} in 1871 and 1872}) and Korteweg and de Vries, respectively (see e.g. \cite{KDV} in 1895); we would also refer readers for related historical review to \cite{Whitham, Miura} by Whitham and Miura, respectively. \\

In this manuscript, we pursue the theory of periodic solution of the initial-boundary-value-problem (\ref{BBM}--\ref{IBVP}) in function space $H^\ell, \ell\ge 1$ on segment $[0,1]$: existence, uniqueness, and stability. The framework is to separating original equation into two auxiliary linear and nonlinear ones. Essentially, the nonlinear term in the equation is treated as a perturbation of the linear problem, and can be estimated by linear theory via interpolation. This approach has been used in \cite{BSZ3, BSZ5} by Bona, Sun, and Zhang on KdV equations posed on segment $[0, 1]$. The KdV reads

\begin{eqnarray*}
u_t + u_{xxx} + u_x + uu_x = f(x,t)
\end{eqnarray*}

\noindent with initial data $\phi$ but nonhomogeneous two-point boundary conditions

\begin{eqnarray*}
u(0,t)=h_1(t), u(1,t)=h_2(t), u_x(1,t)=h_3(t), 
\end{eqnarray*}

\noindent where data are restricted: $(\phi, h_1, h_2, h_3)\in H^s(0, 1)\times H^{s+1\over 3}(0, T)\times H^{s+1\over 3}(0, T)\times H^{s\over 3}(0, T)$ for time $T> 0$. \\

In \cite{BSZ3}, Bona, Sun and Zhang considered the linear KdV prescribed with same conditions:

\begin{eqnarray}\label{KDV-IBVP}
\begin{cases}
v_t + v_{xxx} + v_x = f(x,t)\\
v(x,0)=\phi(x),\\
v(0,t)=h_1(t), v(1,t)=h_2(t), v_x(1,t)=h_3(t). 
\end{cases}
\end{eqnarray}

Authors were able to formulate the solution map of (\ref{KDV-IBVP}): $(\phi, h_1, h_2, h_3)\mapsto v$, 

\begin{eqnarray*}
v(x, t)=W(t)\phi(x) +\int_0^t W(t-s)f(x,s)ds +  W_b(t)(h_1(t), h_2(t), h_3(t))
\end{eqnarray*}

\noindent where $W$ is the semigroup leading to mild solution and $W_b$ is vector form of the boundary integral operator mapping $(h_1, h_2, h_3)$. Local well-posedness on $t\in(0, T)$ in $H^\ell(0, 1), \ell\ge 0,$ can follow from this form. The nonlinear problem is a rewriting of linear IBVP (\ref{KDV-IBVP}) with perturbation $f=uu_x$. $u$'s existence in certain $H^\ell$ can be proved by using fixed point theorem, where the nonlinear term works as a perturbation of the linear problem. More generally and pragmatically, the half-line problem was considered (i.e., quarter plane of $(x, t): x\ge 0, t\ge 0$ referred in earliest paper \cite{BB} by Bona and Bryant, and \cite{BSZ1, BSZ4} by Bona, Sun, Zhang). In the laboratory experiment setting, one end of straight channel full of rest water is applied with mounting wavemaker (at $x=0$), which generates small-amplitude water wave with long wavelength ($u(0,t)=h_1(t)$ with such temporal-periodic $h_1$). The wave will propagate down in the channel to the right toward the other end (``$x=\infty$"), where water has steady state: $u=0$. It is observed that when the small-amplitude periodic force is thrown at $x=0$, the wave turns to become periodic in short term. Two-point IBVP provides a pragmatic approximation of half-line problems as the simulations are implemented on finite segment $[0, L]$. This problem was proposed first in \cite{BD} by Bona and Dougalis, and numerically implemented in \cite{BPS}. Comparison and connection between half-line and two-point boundary problems are studied in \cite{BCSZ} by Bona, Chen, Sun, Zhang. In their presentation, BBM's IBVP about $v(x,t)$ is similar to (\ref{IBVP}) but non-homogeneous at one boundary: $v(0, t)=h_1(t), v(L, t)\equiv 0. $ They proved that $\|u(\cdot, t)-v(\cdot, t)\|_{H^1(0, L)}\le \gamma(t)e^{-L\lambda}$ where $u$ is the solution of half-line problem; $\lambda\in(0,1)$ is selected so that $\gamma$ is a positive function depending on $h_1, \lambda$. Besides, if data $h_1$ is continuous, enlarge $L$ such that $L\rightarrow \infty$, it holds $v(x,t)\rightarrow u(x,t)$. This facilitates the simulation of two-point problem has prediction power on half-line one when $L$ is large. They also studied the connection between IBVP and whole line problem. In their later work \cite{BCSZ2}, the exact solution of IBVP has been given, and the comparison between solutions $v$ to whole line problem and $u_{LR}$ to IBVP is established. Two problems have same initial value $\phi\in H^1,$ but IBVP is posed on $[-L, R]$ with $u_{LR}(-L, t)=g(t), u_{LR}(R, t)=h(t)$ where $(g, h)$ is sufficiently small. It holds that for $\lambda\in(0,1)$, $\|u_{LR}(\cdot, t)-v(\cdot, t)\|_{H^1(-L,R)}\le c_1(t)e^{-\lambda\min\{L, R\}t+c_2t}$, in which $c_1(t)$ is an increasing function and $c_2$ is a constant related to $\phi, \lambda$.\\

It is noteworthy that there were generalized two-point boundary values or forcing problems. For instance, the following generalized boundary condition was studied by Bubnov \cite{Bubnov} on linearized KdV: 

\begin{eqnarray*}
&&\alpha_1 u_{xx}(0,t) + \alpha_2 u_x(0,t) + \alpha_3 u(0,t) = 0,\ \beta_1 u_{xx}(1, t)+\beta_2 u_x(1, t) + \beta_3 u(1, t)=0,\\ &&\gamma_1 u_x(1, t) + \gamma_2 u(1, t) = 0.
\end{eqnarray*}

\noindent Another example, in Bona and Dougalis \cite{BD}, nonhomogeneous two-point boundary condition $u(0, t)=g(t),\,\,u(1, t)=h(t)$ similar as those conditions in (\ref{KDV-IBVP}) for BBM-Burgers equation. Moreover, denote

\begin{eqnarray*}
f(x, t) = xg(t)+(1-x)h(t),
\end{eqnarray*} 

then the error $w=u - f$ satisfies a forced equation:

\begin{eqnarray*}
\begin{cases}  
w_t-w_{xxt}-w_{xx}+w_x + (fw)_x+ww_x = -f_t-ff_x-f_x,\\
w(0,t)=w(1,t)=0,
\end{cases}
\end{eqnarray*}

which leads to same original solution $u(x,t)$.\\

Specifically, analysis on temporally periodic solutions generated by external time-periodic force (forced oscillation) had been also studied via similar semigroup fashion (see, e.g. \cite{BSZ2} by Bona, Sun, Zhang on half-line, and \cite{UsmanZ1, UsmanZ2, WZ, GWX} by Usman, Zhang, and Wang et al, respectively on KdV, viscous Burgers equation, and a 2D hydrodynamical model posed in finite intervals). A further stability question related to large-time dynamics generated by the solutions is asked, due to the fact the water wave evolves to steadily periodic. We might summarize its answer as: \\

\begin{theorem}\label{th1-1}
The periodic solution is unique and globally stable in a phase space $H^\ell, \ell\ge 0$. That is, if the force is time-periodic, the force-generated surface wave turns into time-periodic flow.
\end{theorem}

It is equivalent to view the periodic solution as limit cycle on function space $H^\ell$. For the KdV problems, the periodic solution exists uniquely and the answer to the above question is yes in $H^\ell$ (see e.g \cite{BSZ2, UsmanZ2}). In half-line KdV, the zero-order dissipation term $\alpha u, \alpha > 0,$ must be present in the model to predict the nature of stability that it is observed in experiment in which water wave will be temporally-periodic in a short term, which necessitates the adding dissipation in modelling practice. In corresponding PDE analysis, if dissipation is added, the obtained stability is exponential, alike in the two-point boundary problem. Still in half-line problems of BBM and KdV, \cite{BW} by Bona and Wu discussed the necessity of introducing dissipation terms to stabilize periodic solutions and they found that the viscous term $-\nu u_{xx}$ is not strong as $\alpha u$. Roughly speaking, if only viscous Burgers term is in, the decay is algebraic; however, if $\alpha u$ instead of Burgers' term is in, the decay turns out to be an exponential decay.\\ 

It is sufficiently a historical physically interesting problem when one considers the whole line problems of long-wave models. The BBM can be considered being regularized more by Burgers term, as applying same fashion to ``regularize" the KdV. The related results about BBM-Burgers, KdV-Burgers, and viscous Burgers were discussed in early work \cite{Amick} by Amick, Bona, and Shonbek. On the BBM-Burgers, series of fundamental a priori estimates were prepared, which include the large-time decay behaviours. In the paper, authors already had similar observation on dissipation terms as \cite{BW}: 

$$\|u\|_{L^2}\lesssim t^{-1/4}, \,\,{\rm when}\,\, -\nu u_{xx}\,\,{\rm is\,\, added,}$$ 

while  

$$\|u\|_{L^2}\lesssim e^{-\alpha t},\,\,{\rm when}\,\, \alpha u\,\,{\rm is\,\,added.}$$ 

A later work \cite{GChen} by G. Chen et al proposed the free-vibration Cauchy problem ((\ref{eqmain2}) with $f\equiv 0$ on whole line), and reached existence and exponential stability in $H^1(\mathbb{R})$ as in \cite{Amick}. In model (\ref{eqmain2}) so-called {\it pseudo-parabolic}, there are more regularization terms: smooth functions $\Phi, G$ with respect to $u$, besides Burgers term. In particular, $F$ represents one of generalized nonlinear convection terms including $u^pu_x$, while $\Phi$'s and $G$'s derivative terms are built in to provide stabilization, which agrees with mixed effects of $\alpha u$ and $-\nu u_{xx}$. Close to model's dispersive origin, there are extended models related to theoretical and numerical aspects such as Sobolev–Galpern equation, and we would refer readers to  \cite{GChen} and references therein. \\

Aforementioned references \cite{BD, BL, BCSZ} on BBM equations inspired our current work on two-point boundary problem (\ref{BBM})-(\ref{IBVP}), and \cite{GChen} cushions the ground of the remaining of the manuscript on modified model: (\ref{eqmain2}). The framework we used is classic but extends their results in high function spaces not only limited in $H^1$:

\begin{itemize}
\item We establish the well-posedness and stability results in $H^\ell, \ell\in [1, \infty)$ with elaborate and detailed estimates on Bessel-potential norms of $H^\ell$, given $\phi\in H^\ell$ by probing in $H^2$ and $H^3$, compared to fundamental $H^1$ results seen in previous seminal works \cite{BD, BL, BCSZ} etc in this field.

\item Specifically, we reached the standard contractive semigroup results by using Phillips-Lumer Theorem in $H^2$, and merely energy estimates to reach similar result in $H^3$. We also point out demonstrated for $H^2, H^3$ results, the similar bootstrap argument works for arbitrary $\ell>3$, being similar to that of parabolic equations. 

\end{itemize}

Our paper is organized as follows: Section 2, where we present notations and main theorems of (\ref{BBM}) on existence and local (global) stability of temporally periodic solutions; Section 3, shows the estimates on linearized problem; Section 4 shows the nonlinear estimates and we are able to conclude proof of Theorem \ref{th2-1};  Section 5 addresses the stability of obtained periodic solution; Section 6 extends the discussion in $\cH^\ell, \ell = \{1, 2\}$ of temporally-periodic solution of a pseudo-parabolic version of BBM equation.

\section{Main results}

\ We have norm notations $\|\cdot\|_X$ endowed for standard norm of a classic Banach space $X$. In the following context, $X$ might be of Hilbert: $H^\ell$ or $\cH^\ell$, or that added with smoothing: $Y^\ell_{\tau, T}$, etc. We will have all theorems presented at the end of this section. \\

We also have the following holding through the entire paper:

\begin{itemize}
    \item $\partial_{xx}=\Delta.\,\,(I-\Delta)^{-1}:L^2(0,1)\mapsto L^2(0,1)$ is compact, given the homogeneous two-point boundary condition in (\ref{IBVP}). 
    \item $\|f\|_{\cH^s(0,1)}^2:=\|(I-\Delta)^{\frac{s}{2}}f\|_{L^2(0,1)}^2$.
    \item $\|f\|_{H^k(0,1)}= \sum_{i=0}^k \|\partial^i_xf\|_{L^2(0,1)}$. Note that $H^1_0(0,1)=\{f\in L^2(0,1): f_x\in L^2(0,1), \ f(0)=f(1)=0\}$ is equivalent to  $\cH^1(0,1)$. Also, for $H^1_0(0,1)\cap H^k(0,1)$, the norm $\|u\|_{H^1_0(0,1)\cap H^k(0,1)}:=\|(I-\Delta)^{\frac k2}u\|_{L^2(0,1)}\simeq \|u\|_{H^k(0,1)}$ for any $k\in\N$.
    \item The norm $H^{-1}(0,1)$ is defined to be that of $\cH^{-1}(0,1)$ and the dual $(H^1_0)^*\simeq H^{-1}$.
\end{itemize}

\begin{definition}
    For a dynamical system, we say $u(x,t)$ is locally stable, if $u(t)$ converges to a $\tilde u$ in Banach space $X$ as $t\rightarrow \infty$, when initial value $u_0$ is sufficiently close to $\tilde u$.\\


    We say $u(x,t)$ is globally stable, if $u(t)$ converges to a $\tilde u$ in Banach space $X$ no matter how far the initial value $u_0$ is from $\tilde u$.
\end{definition}

\begin{definition}
    We define the spaces 
    $$Y_{\tau,T}^\ell:=L^{\infty}([\tau,T+\tau]; \cH^\ell(0,1))\cap L^2([\tau,T+\tau]; \cH^\ell(0,1)).$$
    The corresponding norm is defined to be 
    
    \begin{align*}
        \|u\|_{Y^{\ell}_{\tau.T}}^2:=\sup_{\tau\leq t\leq T+\tau}\|u(\cdot,t)\|_{\cH^{\ell}(0,1)}^2+\int_{\tau}^{T+\tau}\|u(\cdot,s)\|_{\cH^{\ell}(0,1)}^2ds.
    \end{align*}
\end{definition}
\begin{remark}
    Throughout this paper, we will use
    $$\frac{1}{M}\|f\|_{H^1_0(0,1)}\leq \|f\|_{\cH^1(0,1)}\leq M\|f\|_{H^1_0(0,1)} $$
    and 
    $$\frac{1}{M}\|f\|_{H^i(0,1)}\leq \|f\|_{\cH^i(0,1)}\leq M\|f\|_{H^i(0,1)}$$
    for all $f\in H^1_0(0,1)\cap H^i(0,1)$ for $i=2,3$.

    Moreover, it can be seen that 
    \begin{align*}
        &\frac{1}{M^2}\bigg(\sup_{\tau\leq t\leq T+\tau}\|u(\cdot,t)\|_{H^{\ell}(0,1)}^2+\int_{\tau}^{T+\tau}\|u(\cdot,s)\|_{H^{\ell}(0,1)}^2ds\bigg)\\
        &\quad \leq \sup_{\tau\leq t\leq T+\tau}\|u(\cdot,t)\|_{\cH^{\ell}(0,1)}^2+\int_{\tau}^{T+\tau}\|u(\cdot,s)\|_{\cH^{\ell}(0,1)}^2ds\\
        &\leq M^2\bigg(\sup_{\tau\leq t\leq T+\tau}\|u(\cdot,t)\|_{H^{\ell}(0,1)}^2+\int_{\tau}^{T+\tau}\|u(\cdot,s)\|_{H^{\ell}(0,1)}^2ds\bigg)
    \end{align*}
\end{remark}

We state our main theorems on (\ref{BBM}) as follows:

\begin{theorem}\label{th2-1}
Let $T, \tau>0$, and $\ell\in [1,\infty)$.
\begin{enumerate}
    \item If $\phi\in \cH^\ell(0,1)$, $f\in L^2([0,\infty);\cH^{\ell-2}(0,1))$, and $\|\phi\|_{\cH^i(0,1)}^2+\|f\|^2_{L^2([0,\infty);\cH^{\ell-2}(0,1))}$ is sufficiently small, then there exists a unique solution $u$ to the equation \eqref{eqmain} and a constant $C>0$ independent of $T$ and $\tau$ such that 
    $$\|u\|_{Y^{\ell}_{0,T}}\leq C(\|\phi\|_{\cH^\ell(0,1)}^2+\|f\|^2_{L^2([0,\infty);\cH^{\ell-2}(0,1))})^{\frac{1}{2}},$$
    and 
    $$\|u\|_{Y^i_{\tau,T}}\leq C(\|\phi\|_{\cH^\ell(0,1)}^2+\|f\|_{L^2([0,\infty);\cH^{\ell-2}(0,1))}^2)^{\frac{1}{2}}.$$
    \item  If $\phi\in \cH^\ell(0,1)$, $f\in L^\infty([0,\infty);\cH^{\ell-2}(0,1))$, and $\|\phi\|_{\cH^\ell(0,1)}^2+\|f\|^2_{L^\infty([0,\infty);\cH^{\ell-2}(0,1))}$ is sufficiently small, then there exists a unique solution $u$ to the equation \eqref{eqmain} and a constant $C>0$ independent of $\tau$ (but dependent on $T$) such that 
    $$\|u\|_{Y^{\ell}_{0,T}}\leq C(\|\phi\|_{\cH^\ell(0,1)}^2+\|f\|^2_{L^\infty([0,\infty);\cH^{\ell-2}(0,1))})^{\frac{1}{2}},$$
    and 
    $$\|u\|_{Y^\ell_{\tau,T}}\leq C(\|\phi\|_{\cH^\ell(0,1)}^2+\|f\|_{L^\infty([0,\infty);\cH^{\ell-2}(0,1))}^2)^{\frac{1}{2}}.$$
\end{enumerate}
\end{theorem}

\begin{theorem}\label{th2-2}
If the conditions in 2 of Theorem \ref{th2-1} hold, and $f$ has temporal-period $\theta$, then the solution $u(x,t)$ has asymptotic temporal-periodicity, i.e., given a positive $T$ and initial time point $\tau,$ there exist positive constants $C$ and $\rho$ such that 
\begin{equation}
\|u(\cdot, \cdot+\theta)-u(\cdot,\cdot)\|_{Y^{\ell}_{\tau, T}}\le C_{\ell}\exp(-\rho\tau)\|u(\cdot,\Gth)-u(\cdot,0)\|_{\cH^\ell(0,1)},
\end{equation}
where $\ell\in [1,\infty)$ and the constant $C_{\ell}$ is independent of $\tau$.
\end{theorem}

\begin{theorem}\label{th2-3}
If the conditions in 2 of Theorem \ref{th2-1} hold, $f$ has period $\theta$ in time, then IBVP (\ref{BBM})--(\ref{IBVP}) possesses a time-periodic solution with period $\theta$ in $\cH^{\ell}(0,1)$, provided that the initial data  $\|\phi\|^2_{\cH^\ell(0,1)}+\|f\|^2_{L^\infty([0,\infty);\cH^{\ell-2}(0,1))}$ is even smaller (depends on $\max\{T,\theta\}$). And this time-periodic solution exhibits local stability. \end{theorem}

\begin{theorem}\label{th2-4}
If $\sup_{t\geq 0}\|f(\cdot,t)\|_{\cH^{\ell-2}(0,1)}$ is sufficiently small (with no restriction on $\|\phi\|_{\cH^{\ell}}$), then the periodic solution $\tilde u(x,t)$ obtained from Theorem \ref{th2-3} for equation (\ref{BBM}) is globally stable in $\cH^{\ell}(0,1)$. That is, any other non-temporally-periodic $u(x,t)$ will exponentially decay towards temporally periodic $\tilde u(x,t)$ in $\mathcal{H}$ other than $L^2$ as $t\rightarrow \infty.$ Additionally, this periodic solution is (globally) unique. 
\end{theorem}

\section{A priori estimates}
In the following sections before Section 6, we consider the initial-boundary-value problem (IBVP) of BBM-Burgers equation:

\begin{equation}\label{eqmain}
\begin{cases}
    u_t+u_x+uu_x-u_{xx}-u_{xxt}=f(x,t),  \quad (x, t)\in [0,1]\times [0, \infty),&\\
    u(x,0)=\phi(x) &\\
    u(0,t)=u(1,t)=0.&
    \end{cases}
\end{equation}
\subsection{A simplified linear problem}
Firstly, consider a simplified linear problem:

\begin{equation} \label{eq1}
\begin{cases}
    u_t+u_x-u_{xx}-u_{xxt}=0, \quad\quad (x, t)\in [0,1]\times [0, \infty)&\\
    u(x,0)=\phi(x), &\\
    u(0,t)=u(1,t)=0. &
\end{cases}
\end{equation}
The solution to \eqref{eq1} can be written as 
\begin{align*}
    (I-\Delta)u_t&=-u_x+u_{xx}\\
    u_t &= (I-\Delta)^{-1}(-u_x+u_{xx}-u+u)=-(I-\Delta)^{-1}\partial_xu-u+(I-\Delta)^{-1}u,
\end{align*}

therefore, 

\begin{equation*}
 u(t)=e^{At}\phi,    
\end{equation*}

where the generator $A$ is defined by $A\psi=-(I-\Delta)^{-1}\partial_x\psi-\psi+(I-\Delta)^{-1}\psi$ and  the domain of $A$ is $H^1_0(0,1)$.

To start estimates on linear problem (\ref{eq1}), we first consider the operator $A$ and its generated $C_0$-semigroup. It holds the following lemma:

\begin{lemma}
    The operator $A$ is bounded on $\cH^\ell(0,1)$ for all $\ell\in \R$, and is dissipative on $\cH^1(0,1)$ and $\cH^2(0,1)$. Moreover, the operator $A$ generates a $C_0$-semigroup and there exists $c>0$ such that 
    $$\|e^{At}\|_{\cH^1(0,1)\mapsto \cH^1(0,1)}\leq e^{-ct}, \quad \|e^{At}\|_{H^1_0(0,1)\mapsto H^1_0(0,1)}\leq Me^{-ct},$$
  and 
  
    $$ \|e^{At}\|_{\cH^2(0,1)\mapsto \cH^2(0,1)}\leq e^{-c''t}, \quad \|e^{At}\|_{H^1_0(0,1)\cap H^2(0,1)\mapsto H^1_0(0,1)\cap H^2(0,1)}\leq M'e^{-ct}. $$
    As a consequence, for any $\ell\in(1,2)$,  there exists $c_{\ell}>0$ such that the interpolation of semigroup holds:
    
     $$ \|e^{At}\|_{\cH^{\ell}(0,1)\mapsto \cH^{\ell}(0,1)}\leq e^{-c_{\ell}t}. $$
\end{lemma}

\begin{proof}

First, using the boundedness of $\partial_x:\cH^\ell(0,1)\to\cH^{\ell-1}(0,1)$ and the embedding of $\cH^{\ell-1}(0,1)$ into $\cH^\ell(0,1)$, 
\begin{align*}
    \|Au\|_{\cH^\ell(0,1)}&\leq \|(I-\Delta)^{\frac{\ell-2}{2}}\partial_xu\|_{L^2(0,1)}+\|(I-\Delta)^{\frac{\ell}{2}}u\|_{L^2(0,1)}+\|(I-\Delta)^{\frac{\ell-2}{2}}u\|_{L^2(0,1)}\leq C\|u\|_{\cH^\ell(0,1)}.
\end{align*}

Note that for $u\in \mathcal{D}(A)$, straight calculation using integration by parts in $L^2$-inner product leads to

\begin{align*}
    \langle Au, u\rangle_{\cH^1(0,1)}=\langle -u_x+u_{xx},u\rangle_{L^2(0,1)}=-\|u_x\|_{L^2(0,1)}^2&\leq -\frac{1}{2}\|u_x\|_{L^2(0,1)}^2-\frac{(c')^2}{2}\|u\|_{L^2(0,1)}^2\\
    &\leq -\frac{1}{2}\min\{(c')^2,1\}\|u\|_{H^1_0(0,1)}^2\\
    &\leq -\frac{M^2}{2}\min\{(c')^2,1\}\|u\|_{\cH^1(0,1)}^2
\end{align*}
where $c'$ is from the Poincar\'e's inequality $\|u_x\|_{L^2(0,1)}\geq c'\|u\|_{L^2(0,1)}$
This shows that $A$ is dissipative in $\cH^1$. \\


On the other hand, consider $Au=g$, that is $u_x-u_{xx}=(I-\Delta)g$ and we can see that 

$$u(x)=e^x\int_0^xe^{-y}\int_0^y(g(z)-g_{zz}(z))dzdy,$$

it is clear that $u\in H^1_0(0,1)\cap H^2(0,1)$ and $\|u\|_{H^1_0(0,1)}\leq C\|g\|_{H^1_0(0,1)}$ (i.e., $\cH^1$). \\

Combining the disspativity and this estimate, by Phillips-Lumer Theorem, there holds exponential decay for semigroup $\{e^{At}\}_{t}: \|e^{At}\|_{\cH^1(0,1)\mapsto \cH^1(0,1)}\leq e^{-ct}, c = -\frac{M}{2}\min\{(c')^2,1\}.$\\

On $\cH^2$'s, note that 
\begin{align*}
    \langle Au,u\rangle_{\cH^{2}(0,1)}&=\langle-u_x+u_{xx},u-u_{xx}\rangle_{L^2(0,1)}\\
    &=\langle u_x,u_{xx}\rangle_{L^2(0,1)}-\|u_{xx}\|_{L^2(0,1)}^2-\|u_{x}\|^2_{L^2(0,1)}\\
    &\leq \frac{1}{2}\|u_x\|_{L^2(0,1)}^2+\frac{1}{2}\|u_{xx}\|^2_{L^2(0,1)}-\|u_{xx}\|_{L^2(0,1)}^2-\|u_{x}\|^2_{L^2(0,1)}\\
    &=-\frac{1}{2}\|u_x\|^2_{L^2(0,1)}-\frac{1}{2}\|u_{xx}\|^2_{L^2(0,1)}\leq -\frac{1}{4}\|u_x\|_{L^2(0,1)}^2-\frac{1}{2}\|u_{xx}\|_{L^2(0,1)}^2-\frac{(c')^2}{4}\|u\|^2_{L^2(0,1)}\\
    &\leq-c''_M\|u\|^2_{\cH^2(0,1)}.
\end{align*}
Thus, we can see that $A$ is also dissipative on $H^2(0,1)\cap H^1_0(0,1)$. We can also see that $A$ has an inverse on $H^1_0(0,1)\cap H^2(0,1)$ similar to above. In fact here, the only thing we need to estimate is the norm of $\|u_{xx}\|_{L^2(0,1)}$. If $Au=g$, it is straight to see that 
\begin{align*}
    \|u_{xx}\|_{L^2(0,1)}&\leq \|(I-\Delta)g\|_{L^2(0,1)}+\|u_x\|_{L^2}\\
    &\leq C\|g\|_{\cH^2(0,1)}+\|u\|_{\cH^1(0,1)}\leq C'\|g\|_{\cH^2(0,1)}.
\end{align*}

Via similar argument in that of $\cH^1$ , we also have  $\{e^{At}\}_{t}: \|e^{At}\|_{\mathcal{D}(A)\mapsto \mathcal{D}(A)}\leq e^{-c''t}.$

\end{proof}

\begin{remark}\label{Re3-2}
 Dissipativity of $A$ in $\cH$ and $\|u\|_{\cH}\lesssim \|g\|_{\cH}$ are equivalent to conditions in Phillips-Lumer Theorem. There are analogous arguments leading to decay semigroups from this aspect in the recent paper \cite{Liu} by Liu, Liu and Zhao and monograph \cite{LZ} by Liu and Zheng on dissipative semigroups cited therein.\\
\end{remark}

\begin{remark}
We first note that the solution is continuous in time because it is of the form $e^{At}\phi$ and $\{ e^{At}\}_t$ is a $C_0$-semigroup. {We also consider the semigroup generated by $A$ on $H^1_0(0,1)\cap H^2(0,1)$ to obtain a similar result without much difference so that higher order regularity ($\ell>3$) can be obtain without difficulty.}
\end{remark}

\begin{lemma}[A priori $H^1$estimate]\ \label{aprih1}\\
Let $\phi \in \cH^1$. Then, we have the following estimate for all $t>0$
    $$\|u(\cdot,t)\|_{\cH^1(0,1)}^2\leq C(\|\phi\|_{L^2(0,1)}^2+\|\phi_x\|_{L^2(0,1)}^2)e^{-\min\{c^2,1\}t},$$
    where $c$ is the majorizing constant depending on the spatial domain for the Poincar\'e's inequality.\\
    
     It follows that there exists $C>0$ such that for $T> 0$, 
    $$\|u\|_{L^{\infty}([0,T];\cH^{1}(0,1))} + \|u\|_{L^2([0,T];\cH^{1}(0,1))}\leq C\|\phi\|_{\cH^{1}(0,1)}$$
and 
$$\|u_t\|_{L^{\infty}([0,T];\cH^{1}(0,1))} + \|u_t\|_{L^2([0,T];\cH^{1}(0,1))}\leq C\|\phi \|_{\cH^{1}(0,1)}.$$
\end{lemma}
\begin{proof}
Multiply both sides by $u$ and then integrate with respect to $L^2(0,1)$, at time $t$ we have

\begin{align*}
    \int_0^1 u_tudx + \int_0^1 u_xudx-\int_{0}^1u_{xx}udx-\int_0^1u_{xxt}udx &=0, \\
    \frac{1}{2}\frac{d}{dt}\|u\|_{L^2(0,1)}^2+\frac{1}{2}[(u(1,t))^2-(u(0,t))^2]+\|u_x\|_{L^2(0,1)}^2+\frac{1}{2}\frac{d}{dt}\|u_x\|_{L^2(0,1)}^2&=0, \\
  \frac{d}{dt}\|u\|_{L^2(0,1)}^2+2\|u_x\|_{L^2(0,1)}^2+\frac{d}{dt}\|u_x\|_{L^2(0,1)}^2&=0.
\end{align*}

We shall note that $\int_0^1 u_{xtx}udx=[u(1,t)u_{xt}(1,t)-u(0,t)u_{xt}(0,t)]-\int_0^1 u_{xt}u_xdx=-\int_0^1 u_{xt}u_xdx$ owing to the boundary conditions provided in (\ref{IBVP}).\\

By Poincar\'e's inequality, i.e., $\|u_x(\cdot,s)\|_{L^2(0,1)}\geq c\|u(\cdot,s)\|_{L^2(0,1)}$ , at time $t$, we have

\begin{align*}
 (\|u_x\|_{L^2(0,1)}^2+c^2\|u\|_{L^2(0,1)}^2)+   \frac{d}{dt}\|u\|_{L^2(0,1)}^2+\frac{d}{dt}\|u_x\|_{L^2(0,1)}^2&\leq 0 \\
  \frac{d}{dt}\bigg(\|u\|_{L^2(0,1)}^2+\|u_x\|_{L^2(0,1)}^2\bigg)\leq -\min\{c^2,1\}(\|u\|_{L^2(0,1)}^2+\|u_x\|_{L^2(0,1)}^2).
\end{align*}

Therefore, by Gronwall's lemma, 
$$\|u(t)\|_{L^2(0,1)}^2+\|u_x(t)\|_{L^2(0,1)}^2\leq (\|u(\cdot,0)\|_{L^2(0,1)}^2+\|u_x(\cdot,0)\|_{L^2(0,1)}^2)e^{-\min\{c^2,1\}t}.$$
Thus, 
$$\|u(t)\|_{\cH^1(0,1)}^2\leq 2M(\|u(\cdot,0)\|_{L^2(0,1)}^2+\|u_x(\cdot,0)\|_{L^2(0,1)}^2)e^{-\min\{c^2,1\}t}.$$

We have the desired estimate by taking $L^{\infty}([0,T])$ and $L^2([0,T])$ with respect to $t$.\\

Let $v:=u_t$ hence $v=Au =Ae^{At}\phi$.  Then we can see that $v_t=Av$ and it solves
\begin{equation*}
\begin{cases}
    v_t+v_x-v_{xx}-v_{xxt}=0, &\\
    v(x,0)= A\phi(x),&\\ 
    v(0,t)=v(1,t)=0.&
\end{cases}
\end{equation*}

Using a priori estimate we have 

$$\|u_t\|_{L^{\infty}([0,T];\cH^{1}(0,1))} + \|u_t\|_{L^2([0,T];\cH^{1}(0,1))}\leq C\|A\phi\|_{\cH^1(0,1)}\leq C'\|\phi \|_{\cH^1(0,1)}.$$
\end{proof}

\begin{lemma}[Estimate for $\ell=2$:]\label{aprih2}\ \\
    Let $\phi \in \cH^2(0,1)$. Then, for all $t>0$, 
    $$\|u(\cdot,t)\|_{\cH^2(0,1)}^2\leq e^{-2c''t}\|\phi\|_{\cH^2(0,1)}^2. $$
     It follows that there exists $C>0$ such that for $T> 0$, 
    $$\|u\|_{L^{\infty}([0,T];\cH^{2}(0,1))} + \|u\|_{L^2([0,T];\cH^{2}(0,1))}\leq C\|\phi\|_{\cH^2(0,1)}$$
and 
$$\|u_t\|_{L^{\infty}([0,T];\cH^{2}(0,1))} + \|u_t\|_{L^2([0,T];\cH^{2}(0,1))}\leq C\|\phi \|_{\cH^2(0,1)}.$$
\end{lemma}
\begin{proof}
   It suffices to note that 
    $$\|u(t)\|_{\cH^{2}(0,1)}\leq \|e^{At}\|_{\cH^2(0,1)\to \cH^2(0,1)}\|\phi\|_{\cH^2(0,1)}\leq e^{-c''t}\|\phi\|_{\cH^2(0,1)}.$$

Rest of estimates can be obtained by similar calculation given the exponential decay as $\ell=1$.
\end{proof}

\begin{lemma}[Estimate for $\ell=3$:]\label{aprih3}\ \\
    Let $\phi \in \cH^3(0,1)$. Then, for all $t>0$, 
    $$\|u(\cdot,t)\|_{\cH^3(0,1)}^2\leq Ce^{-c'''t}\|\phi\|_{\cH^3(0,1)}^2. $$
     It follows that there exists $C>0$ such that for $T> 0$, 
    $$\|u\|_{L^{\infty}([0,T];\cH^{3}(0,1))} + \|u\|_{L^2([0,T];\cH^{3}(0,1))}\leq C'\|\phi\|_{\cH^3(0,1)}$$
and 
$$\|u_t\|_{L^{\infty}([0,T];\cH^{3}(0,1))} + \|u_t\|_{L^2([0,T];\cH^{3}(0,1))}\leq C'\|\phi \|_{\cH^3(0,1)}.$$
\end{lemma}
\begin{proof}
Here we will differentiate \eqref{eq1} once with respect to $x$ and we have 
$$u_{xt}+u_{xx}-u_{xxx}-u_{xxxt}=0.$$
Then, by multiplying $u_{xxx}$ and then integrate over $x\in (0,1)$, 
\begin{align*}
    \frac{d}{dt}\|u_{xxx}\|_{L^2(0,1)}^2+\|u_{xxx}\|_{L^2(0,1)}^2&=\int_0^1-2(v_{xt}+u_{xx})u_{xxx}dx \\
     \frac{d}{dt}\|u_{xxx}\|_{L^2(0,1)}^2+\|u_{xxx}\|_{L^2(0,1)}^2 &\leq 2C_{c,c''}\|u_{xt}+u_{xx}\|_{L^2}^2+(1-\min\{1/2,(c)^2,2c''\})\|u_{xxx}\|_{L^2(0,1)}^2 \\
    \frac{d}{dt}\|u_{xxx}\|_{L^2(0,1)}^2 +\min\{1/2,(c)^2,2c''\}\|u_{xxx}\|_{L^2(0,1)}^2&\leq 4C_{c,c''}\|u_t\|_{H^1(0,1)}^2+4C_{c,c''}\|u\|_{H^2(0,1)}^2. 
\end{align*}
Therefore, we can conclude that by Lemmas \ref{aprih1} and \ref{aprih2},
\begin{align*}
    e^{\min\{1/2,(c)^2,2c''\}t}\|u_{xxx}\|_{L^2(0,1)}^2-\|\phi_{xxx}\|_{L^2(0,1)}^2&\leq 4\int_0^te^{\min\{1/2,(c)^2,2c''\}(s)}(\|u_t\|_{H^1(0,1)}^2+\|u\|_{H^2(0,1)}^2)ds \leq C\|\phi\|_{H^2(0,1)}^2,
\end{align*}
which implies 
$$\|u\|_{H^3(0,1)}^2\leq Ce^{-\min\{1/2,(c)^2,2c''\}t}\|\phi\|_{H^3(0,1)}^2.$$
Thus, we can also rewrite as
$$\|u\|_{\cH^3(0,1)}^2\leq CM^4e^{-\min\{1/2,(c)^2,2c''\}t}\|\phi\|_{\cH^3(0,1)}^2.$$
The estimate for $v=u_{t}$ can be done similarly with the boundedness of $A$ on $\cH^\ell(0,1)$.
\end{proof}

\subsection{A linear equation with force}

We consider the equation with imposed force in the domain  prescribed with zero boundary condition:

\begin{equation} \label{eq2}
\begin{cases}
    u_t+u_x-u_{xx}-u_{xxt}=f(x,t) \quad\quad (x,t)\in [0,1]\times [0, \infty),&\\
    u(x,0)=0, &\\
    u(0,t)=u(1,t)=0. &
\end{cases}
\end{equation}

Using semigroup as before, we know the solution to \eqref{eq2} is

$$u(x,t):=\int_0^te^{A(t-s)}[(I-\Delta)^{-1}f(s)]ds.$$

 We would derive the linear estimates in $Y^\ell_{0, T}$ with $\ell=1, 2, 3$, sequentially.\\

\begin{lemma}[A priori estimate for \eqref{eq2}, $\ell=1$]\ \label{apri2l1} \\
Assuming $f\in L^2([0,T];\cH^{-1}(0,1))$, then 
   $$ \|u\|_{L^{\infty}([0,T];\cH^1(0,1))}^2+\|u\|_{L^{2}([0,T];\cH^1(0,1))}^2\leq  C\int_0^T\|f(s)\|^2_{\cH^{-1}(0,1)}ds. $$
   Moreover, we have 
   $$ \|u\|_{W^{1,\infty}([0,T];\cH^1(0,1))}^2+\|u\|_{H^{1}([0,T];\cH^1(0,1))}^2\leq  C\bigg(\|f(\cdot,0)\|_{\cH^{-1}(0,1)}^2+\int_0^T\|f(\cdot,s)\|^2_{\cH^{-1}(0,1)}ds\bigg). $$
\end{lemma}
\begin{proof}

Multiply $u$ onto the equation and integrate over $[0,1]$,

    \begin{align*}
    \int_0^1 u_tudx + \int_0^1 u_xudx-\int_{0}^1u_{xx}udx-\int_0^1u_{xxt}udx &=\int_0^1f(x,t)u(x,t)dx \\
   \frac{d}{dt}\|u\|_{L^2(0,1)}^2+2\|u_x\|_{L^2(0,1)}^2+\frac{d}{dt}\|u_x\|_{L^2(0,1)}^2&=2\int_0^1f(x,t)u(x,t)dx \\
    \frac{d}{dt}\|u\|_{L^2(0,1)}^2+2\|u_x\|_{L^2(0,1)}^2+\frac{d}{dt}\|u_x\|_{L^2(0,1)}^2&=2\int_0^1f(x,t)u(x,t)dx \\
     (\|u_x\|_{L^2(0,1)}^2+c^2\|u\|_{L^2(0,1)}^2)+   \frac{d}{dt}\|u\|_{L^2(0,1)}^2+\frac{d}{dt}\|u_x\|_{L^2(0,1)}^2&\leq 2\int_0^1f(x,t)u(x,t)dx \\
    (\|u_x\|_{L^2(0,1)}^2+c^2\|u\|_{L^2(0,1)}^2)+   \frac{d}{dt}\|u\|_{L^2(0,1)}^2+\frac{d}{dt}\|u_x\|_{L^2(0,1)}^2&\leq C\|f(\cdot,t)\|_{\cH^{-1}(0,1)}^2 + \frac{c^2}{2}\|u(t)\|^2_{\cH^1(0,1)},
\end{align*} whose last line is applied with duality of $\cH^1$ and $\cH^{-1}$.\\

Whence,

\begin{align*}
\frac{d}{dt}\|u(\cdot,t)\|_{L^2(0,1)}^2+\frac{d}{dt}\|u_x(\cdot,t)\|_{L^2(0,1)}^2 &\leq C\|f(\cdot,t)\|^2_{\cH^{-1}(0,1)}-\frac{c^2}{2}(\|u(\cdot,t)\|_{L^2(0,1)}^2+\|u_x(\cdot,t)\|_{L^2(0,1)}^2),\\
\|u(\cdot,t)\|_{L^2(0,1)}^2+\|u_x(\cdot,t)\|_{L^2(0,1)}^2&\leq  C\int_0^te^{-\frac{c^2}{2}(t-s)}\|f(s)\|^2_{\cH^{-1}(0,1)}ds.
\end{align*}
If we take supremum over $t\in [0,T]$, we have 
\begin{align*}
\|u\|_{L^{\infty}([0,T];H^1(0,1))}^2&\leq  C\int_0^T\|f(s)\|^2_{\cH^{-1}(0,1)}ds,\quad \text{and}\quad  \|u\|_{L^{\infty}([0,T];\cH^1(0,1))}^2\leq  CM^2\int_0^T\|f(s)\|^2_{\cH^{-1}(0,1)}ds;
\end{align*}
and if we integrate with resepct to $t\in [0,T]$, we have 
\begin{align*}
\|u\|_{L^{2}([0,T];H^1(0,1))}^2&\leq  C'\int_0^T\|f(s)\|^2_{\cH^{-1}(0,1)}ds,\quad \text{and}\quad \|u\|_{L^{2}([0,T];\cH^1(0,1))}^2\leq  CM^2\int_0^T\|f(s)\|^2_{\cH^{-1}(0,1)}ds.
\end{align*}
The proof is completed. 
\end{proof}
Since $u_t=Au+(I-\Delta)^{-1}f(x,t)$, $v=u_t$ solves
\begin{equation*}
\begin{cases}
    v_t+v_x-v_{xx}-v_{xxt}=f(x,t), &\\
    v(x,0)= (I-\Delta)^{-1}f(x,0),&\\ 
    v(0,t)=v(1,t)=0.&
\end{cases}
\end{equation*}
Then, $\ds v(t)(x)=e^{At}(I-\Delta)^{-1}f(x,0)+\int_0^{t}e^{A(t-s)}(I-\Delta)^{-1}f(x,s)ds$. Therefore, 
\begin{align*}
    \|v\|_{\cH^1(0,1)}&\leq \|e^{At}(I-\Delta)^{-1}f(x,0)\|_{\cH^1(0,1)}+ \int_0^t \|e^{A(t-s)}(I-\Delta)^{-1}f(x,s)\|_{\cH^1(0,1)}ds \\
    &\leq e^{-ct}\|f(\cdot,0)\|_{\cH^{-1}(0,1)}+\int_0^te^{-c(t-s)}\|f(x,s)\|_{\cH^{-1}(0,1)}ds.
\end{align*}
\begin{lemma}[A priori estimate for \eqref{eq2}, $\ell=2$]\ \label{apri2l2} \\
Assuming $f\in L^2([0,T];L^2(0,1))$, then 
   $$ \|u\|_{L^{\infty}([0,T];\cH^2(0,1))}^2\leq  C\int_0^T\|f(s)\|^2_{L^2(0,1)}ds. $$
\end{lemma}
\begin{proof}
    It suffices to note that 
    \begin{align*}
        \|u(\cdot,t)\|_{\cH^2(0,1)}&\leq \int_0^{t}\|e^{A(t-s)}\|_{\cH^2(0,1)\to\cH^2(0,1)}\|(I-\Delta)^{-1}f(s)\|_{\cH^2(0,1)}ds\\
        &\leq  \int_0^{t}M e^{-c''(t-s)}\|f(s)\|_{L^2(0,1)}ds.
    \end{align*}
    The estimates follows the calculation in proof of Lemma \ref{apri2l1}.
\end{proof}

\begin{lemma}[A priori estimate for \eqref{eq2}, $\ell=3$]\ \label{apriforceh3}\\
Assuming $f\in L^2([0,T];H^1(0,1))$, then 
   $$ \|u\|_{L^{\infty}([0,T];\cH^3(0,1))}^2+\|u\|_{L^{2}([0,T];\cH^3(0,1))}^2\leq  C\int_0^T\|f(s)\|^2_{H^{1}(0,1)}ds. $$
\end{lemma}
\begin{proof}
Following the idea of Lemma \ref{aprih3}, we again differentiate \eqref{eq2} with respect to $x$, then multiply both sides by $u_{xxx}$, and integrate over $x\in (0,1)$, we have at time $t$
\begin{align*}
\frac{d}{dt}\|u_{xxx}\|_{L^2(0,1)}+\|u_{xxx}\|_{L^2(0,1)}^2&=2\int_0^1(u_{xt}+u_{xx}-f_{x}(x,t))u_{xxx}dx\\
&\leq 8(\|u_{t}\|^2_{H^1(0,1)}+\|u\|^2_{H^2(0,1)}+\|f\|_{H^1(0,1)}^2)+\frac{1}{2}\|u_{xxx}\|^2_{L^2(0,1)}.
\end{align*}
We can conclude the stated estimate using results of $l=1, 2$: Lemmas \ref{apri2l1} and \ref{apri2l2} to see that
$$\|u(t)\|_{\cH^3(0,1)}^2\leq C'\int_0^t e^{-2(t-s)}\|f(s)\|_{H^1(0,1)}^2ds,\quad \text{thus}\quad  \|u(t)\|_{\cH^3(0,1)}^2\leq C'M\int_0^t\|f(s)\|_{H^1(0,1)}^2ds. $$
\end{proof}

\begin{remark}
When $\ell=3$, the dissipativity of generating operator $A$ is not obvious to obtain following fashion of $\cH^2$, hence instead of Hille-Yosida type of theorems, we use the classic energy technique to obtain the estimate; it is similar as ``bootstrap" argument used in regularity estimates of parabolic equations (see e.g. \cite{Qin}). If two-point boundary conditions on derivatives such as $f_x(0, t)=f_x(1, t)=0$ are imposed, dissipativity of $A$ in higher space $\cH^{\ell}, \ell>2$ and related classic results can follow from differentiated equation and  argument as Remark \ref{Re3-2}.
\end{remark}

\subsection{A bilinear estimate}
In this subsection, we will establish a bilinear estimate with constant independent of $T$.
\begin{lemma}\label{bilinearest}
    Let $\ell\geq 1$ and $0\leq \ell'\leq \ell-1$. Suppose $v,w\in L^{\infty}([0,T]; H^\ell(0,1))\cap L^2([0,T]; H^\ell(0,1))$. Then, for all $t\in [0,T]$, we have 
     \begin{align*}
        \int_0^t\|(vw)_x\|_{\cH^{\ell'}(0,1)}^2dt\leq C'' \|v\|_{Y^{\ell}_{0,t}}^2\|w\|_{Y^{\ell}_{0,t}}^2,
    \end{align*}
    where constant $C^{''}$ doesn't depend on $T$.
\end{lemma}
\begin{proof}
    We first show for the $L^2$ case. Following the proof of Lemma 3.1 in \cite{BSZ2}, by Gagliardo–Nirenberg interpolation inequality,
    \begin{align*}
        \|vw_x\|_{L^2(0,1)}^2&\leq \|v\|_{L^{\infty}(0,1)}^2\|w_x\|_{L^2(0,1)}^2 \\
        &\leq 4(\|v_x\|_{L^2(0,1)}\|v\|_{L^2(0,1)}+C\|v\|_{L^2(0,1)}^2)\|w_x\|_{L^2(0,1)}^2.
    \end{align*}
    We first consider $$\int_0^t\|v\|_{L^2(0,1)}^2\|w_x\|_{L^2(0,1)}^2ds.$$
    Note that one can bound $\sup_{t\in [0,T]}\|v\|_{L^2(0,1)}^2\leq \|v\|^2_{L^{\infty}([0,T];\cH^\ell(0,1))}.$ Therefore, we have 
    \begin{align*}
        \int_0^t\|v\|_{L^2(0,1)}^2\|w_x\|_{L^2(0,1)}^2ds&\leq  \|v\|^2_{L^{\infty}([0,t];\cH^\ell(0,1))}\|w\|_{L^2([0,t];\cH^{\ell}(0,1))}^2
    \end{align*}
    using the fact that $\partial_x:\cH^{\ell}(0,1)\mapsto \cH^{\ell-1}(0,1)$ is bounded (with constant 1).\\

    Meanwhile, by Young's inequality for products,
    \begin{align*}
        &\int_0^t \|v_x\|_{L^2(0,1)}\|v\|_{L^2(0,1)}\|w_x\|_{L^2(0,1)}^2ds\leq  \|v\|_{L^{\infty}([0,t];\cH^{\ell}(0,1))}^2\|w\|_{L^2([0,t];\cH^{\ell}(0,1))}^2.
    \end{align*}
    Therefore, we have 
    $$\int_0^t\|vw_x\|_{L^2(0,1)}^2ds\leq C'\|v\|_{L^{\infty}([0,t];\cH^{\ell}(0,1))}^2\|w\|_{L^2([0,t];\cH^{\ell}(0,1))}^2. $$

    Interchanging the role of $v$ and $w$, we have 
    \begin{align*}
        \int_0^t\|(vw)_x\|_{L^2(0,1)}^2\leq C'' (\|v\|_{L^{\infty}([0,t];\cH^{\ell}(0,1))}^2+\|v\|_{L^{2}([0,t];\cH^{\ell}(0,1))}^2)(\|w\|_{L^{\infty}([0,t];\cH^{\ell}(0,1))}^2+\|w\|_{L^{2}([0,t];\cH^{\ell}(0,1))}^2).
    \end{align*}

For general $\ell'\in \N$ and $\ell'$ in between two integers, we can simply follow the argument as the proof of Lemma 3.1 in \cite{BSZ2}. 
\end{proof}
\section{Well-posedness of nonlinear problem }
This section will be separated into two parts. We will focus on the case $\ell=1, 2$ in the first part; in the second part, we will discuss the case $\ell=3$; the proof of Theorem \ref{th2-1}, our first main theorem, will be in the last part.\\

 There are the linear IBVP

\begin{equation}\label{eqsub1}
\begin{cases}
    v_t+v_x-v_{xx}-v_{xxt}=f(x,t),  \quad (x, t)\in [0,1]\times [0, \infty),&\\
    v(x,0)=\phi(x) &\\
    v(0,t)=v(1,t)=0.&
    \end{cases}
\end{equation}
and nonlinear IBVP

\begin{equation}\label{eqsub2}
\begin{cases}
    w_t+w_x-w_{xx}-w_{xxt}+ww_x+(vw)_x=-vv_x,  \quad (x, t)\in [0,1]\times [0, \infty),&\\
    w(x,0)=0&\\
    w(0,t)=w(1,t)=0.&
    \end{cases}
\end{equation}

It is clear that $v(x,t)+w(x,t)$ solves the original IBVP of \eqref{eqmain}. \\

Note that the mild solution to \eqref{eqsub1} is 

$$v(t)=e^{At}\phi+\int_0^t e^{A(t-s)}(I-\Delta)^{-1}f(s)ds;$$
and the mild solution to \eqref{eqsub2} is
$$w(t)=-\int_0^t e^{A(t-s)}(I-\Delta)^{-1}[w(s)w_x(s)+(v(s)w(s))_x+v(s)v_x(s)]ds.$$

From the discussion in the previous section, we are able to summary to get the following proposition before we start the estimate of (\ref{eqsub2}):  

\begin{prop}[Existence of Solution to \eqref{eqsub1}]\label{existsub1}\ \\
    Let $\phi \in \cH^i(0,1)$ and $f\in L^{2}([0,T];\cH^{i-2}(0,1))$ for $i\in\{1,2,3\}$. Then,
    \begin{align*}
        \|v\|_{Y^{i}_{0,T}}+\|v_t\|_{Y^{i}_{0,T}}\leq C(\|\phi\|_{\cH^i(0,1)}^2+\|f\|^2_{L^2([0,T];\cH^{i-2}(0,1))})^{\frac{1}{2}},
    \end{align*}
    where $i\in \{1,2,3\}$ and the constant $C$ is independent of $T$.
\end{prop}
\subsection{The cases $\ell=1$ and $\ell=2$}
In this subsection, we will establish some bounds of $w$ with $\ell=1$ and $\ell=2$. The argument for $\ell=1$ and $\ell=2$ are formatted same by using semigroup estimates, so in the proof we will only provide that of $\ell=1$. 

\begin{prop}[Existence of Solution to \eqref{eqsub2}]\label{existsub2}\ \\
    Let $\phi \in \cH^1(0,1)$ and $f\in L^2([0,T];\cH^{-1}(0,1))$. Then, there exists $c^{(1)}<1$ (independent of $T$) such that if $(\|\phi\|_{\cH^1(0,1)}^2+\|f\|^2_{L^2([0,T];\cH^{-1}(0,1))})^{\frac12}<c^{(1)}$ then \eqref{eqsub2} has a unique solution $w$. Moreover, in this case,  $$\|w\|_{Y^{1}_{0,T}}\leq (\|\phi\|^2_{\cH^1(0,1)}+\|f\|^2_{L^2([0,T];\cH^{-1}(0,1))})^{\frac12}.$$

    If $\phi \in \cH^2(0,1)$ and $f\in L^2([0,T];L^2(0,1))$, then there exists $c^{(2)}\leq c^{(1)}$ (independent of $T$) such that if  $(\|\phi\|_{\cH^2(0,1)}^2+\|f\|^2_{L^2([0,T];L^2(0,1))})^{\frac12}<c^{(2)}$ then $$\|w\|_{Y^{2}_{0,T}}\leq (\|\phi\|^2_{\cH^2(0,1)}+\|f\|^2_{L^2([0,T];L^2(0,1))})^{\frac12}.$$
\end{prop}
\begin{proof}
    To prove the existence and uniqueness of the solutions of $w$, we shall use contraction mapping. We set 
    $$\Gamma(w)(t):=-\int_0^t e^{A(t-s)}(I-\Delta)^{-1}[w(s)w_x(s)+(vw)_x(s)+v(s)v_x(s)]ds. $$
    We shall show that $\Gamma:Y_{0,T,M'}^\ell\mapsto Y_{0,T,M'}^1$, where $ Y_{0,T,M'}^1:=\{w\in  Y_{0,T}^1:\|w\|_{ Y_{0,T}^1}\leq M'\}$. It is natural to take $M':= K(\|\phi\|_{\cH^{1}(0,1)}^2+\|f\|^2_{L^2([0,T];\cH^{-1}(0,1))})^{\frac12}$ and we will determine $K$ later.\\

Let $w\in Y_{0,T,M'}^1$, then
\begin{align*}
    \|\Gamma(w)(t)\|_{\cH^1(0,1)}&=\bigg\|\int_0^t e^{A(t-s)}(I-\Delta)^{-1}[(w(s)+v(s))(w(s)+v(s))_x]ds\bigg\|_{\cH^1(0,1)} \\
    &\leq \int_0^t \|e^{A(t-s)}\|_{\cH^1(0,1)\to \cH^1(0,1)}\bigg(\frac{1}{2}\|(ww)_x\|_{\cH^{-1}(0,1)}+\|(vw)_x\|_{\cH^{-1}(0,1)}+\frac{1}{2}\|(vv)_x\|_{\cH^{-1}(0,1)}\bigg)ds\\
    &\leq  \int_0^t e^{-c(t-s)}\bigg(\frac{1}{2}\|(ww)_x\|_{\cH^{-1}(0,1)}+\|(vw)_x\|_{\cH^{-1}(0,1)}+\frac{1}{2}\|(vv)_x\|_{\cH^{-1}(0,1)}\bigg)ds\\
    &\leq \bigg[ \int_0^t e^{-2c(t-s)}ds\bigg]^{\frac12}\bigg(2^3 \int_0^t\bigg(\frac{1}{2}\|(ww)_x\|^2_{\cH^{-1}(0,1)}+\|(vw)_x\|^2_{\cH^{-1}(0,1)}+\frac{1}{2}\|(vv)_x\|^2_{\cH^{-1}(0,1)}\bigg)ds\bigg)^{\frac12}\\
     &\leq C_{A}\bigg[\bigg(\int_0^t\|(ww)_x(s)\|^2_{L^2(0,1)}ds\bigg)^{\frac12}+\bigg(\int_0^t\|(vw)_x(s)\|^2_{L^2(0,1)}ds\bigg)^{\frac12}+\bigg(\int_0^t\|(vv)_x(s)\|^2_{L^2(0,1)}ds\bigg)^{\frac12}\bigg]
\end{align*}

Here we have used the fact that $\|u\|_{\cH^{-1}(0,1)}\leq \frac{1}{\sqrt{2}}\|u\|_{L^2(0,1)}$.\\ 

Using Lemma \ref{bilinearest}, we have 
\begin{align*}
     \|\Gamma(w)(t)\|_{\cH^1(0,1)}&\leq C_A'(\|w\|_{Y^{1}_{0,t}}^2+\|w\|_{Y^{1}_{0,t}}\|v\|_{Y^{1}_{0,t}}+\|v\|_{Y^{1}_{0,t}}^2)
\end{align*}

thus
\begin{align*}
    \sup_{0\leq t\leq T} \|\Gamma(w)(t)\|_{\cH^1(0,1)}&\leq C_A'' (\|w\|_{Y^{1}_{0,T}}^2+\|v\|_{Y^{1}_{0,T}}^2)\\
    &\leq C_A''\bigg(\|w\|_{Y^{1}_{0,T}}^2+\|\phi\|^2_{\cH^{1}(0,1)}+\|f\|_{L^2([0,T];\cH^{-1}(0,1))}^2\bigg)
\end{align*}

To obtain the estimate for $L^2([0,T];\cH^{1}(0,1))$, notice that one has 
\begin{align*}
   & \int_0^T\|\Gamma(w)(t)\|_{\cH^{1}(0,1)}^2dt\\
   &\leq \int_0^T  \bigg(\int_0^t \|e^{A(t-s)}\|_{\cH^1(0,1)\to \cH^1(0,1)} \bigg(\frac{1}{2}\|(ww)_x\|_{\cH^{-1}(0,1)}+\|(vw)_x\|_{\cH^{-1}(0,1)}+\frac{1}{2}\|(vv)_x\|_{\cH^{-1}(0,1)}\bigg)ds\bigg)^2dt \\
   &\leq \int_0^T \|e^{At}\|_{\cH^1(0,1)\to \cH^1(0,1)}dt \int_0^T\bigg(\frac{1}{2}\|(ww)_x\|_{\cH^{-1}(0,1)}+\|(vw)_x\|_{\cH^{-1}(0,1)}+\frac{1}{2}\|(vv)_x\|_{\cH^{-1}(0,1)}\bigg)^2dt \\
   &\leq C_A \bigg[\int_0^t\|(ww)_x(s)\|^2_{L^2(0,1)}ds+\int_0^t\|(vw)_x(s)\|^2_{L^2(0,1)}ds+\int_0^t\|(vv)_x(s)\|^2_{L^2(0,1)}ds\bigg]
\end{align*}
by the aid of Young's convolution inequality.\\

Therefore, using similar argument as above, we can conclude that 
$$\|\Gamma(w)\|_{Y^{1}_{0,T}}\leq C'''_A \bigg(\|w\|_{Y^{1}_{0,T}}^2+\|\phi\|^2_{\cH^{1}(0,1)}+\|f\|_{L^2([0,T];\cH^{-1}(0,1))}^2\bigg)\leq C_A'''((M')^2+\frac{(M')^2}{K^2}), $$
and 
$\Gamma(w)\in Y^{1}_{0,T,M'}$ if $M'<\frac{1}{C'''_A}\frac{K^2}{K^2+1}$.\\

Next, we shall show that $\|\Gamma(w)-\Gamma(w')\|_{Y^{1}_{0,T}}\leq \frac{1}{2}\|w-w'\|_{Y^{1}_{0,T}}$ provided that $\|w\|_{Y^{1}_{0,T}},\|w'\|_{Y^{1}_{0,T}}$ are small.\\

First observe that 
\begin{align*}
    &\Gamma(w)(t)-\Gamma(w')(t)\\&= -\int_0^t e^{A(t-s)}(I-\Delta)^{-1}[w(s)w_x(s)+(vw)_x(s)-w'(s)w'_x(s)-(vw')_x(s)]ds \\
    &= -\int_0^t e^{A(t-s)}(I-\Delta)^{-1}[(w(s)[w_x-w'_x](s)+(w-w')(s)w_x'(s))+(v[w-w'])_x(s)]ds.
\end{align*}

Thus, we have 
\begin{align*}
    &\sup_{0\leq t\leq T}\|\Gamma(w)(t)-\Gamma(w')(t)\|_{\cH^{1}(0,1)}\\
    &\leq C_A\bigg[\bigg(\int_0^t\|w(s)[w_x-w'_x](s)\|^2_{\cH^{-1}(0,1)}ds\bigg)^{\frac12}+\bigg(\int_0^t\|w'_x(s)[w-w'](s)\|^2_{\cH^{-1}(0,1)}ds\bigg)^{\frac12}\\
    &\quad\quad\quad +\bigg(\int_0^t\|(v[w-w'])_x(s)\|^2_{\cH^{-1}(0,1)}ds\bigg)^{\frac12}\bigg].
\end{align*}
From Lemma \ref{bilinearest}, we have 
\begin{align*}
    \int_0^t\|w(s)[w_x-w'_x](s)\|^2_{\cH^{-1}(0,1)}ds&\leq C'' \|w\|^2_{Y^{1}_{0,T}}\|w-w'\|^2_{Y^{1}_{0,T}}\\
    \int_0^t\|w'_x(s)[w-w'](s)\|^2_{\cH^{-1}(0,1)}ds&\leq C''\|w'\|^2_{Y^{1}_{0,T}}\|w-w'\|^2_{Y^{1}_{0,T}}\\
    \int_0^t\|(v[w-w'])_x(s)\|^2_{\cH^{-1}(0,1)}ds&\leq C''\|v\|_{Y^{1}_{0,T}}^2\|w-w'\|^2_{Y^{1}_{0,T}}.
\end{align*}
Therefore, 
\begin{align*}
    &\sup_{0\leq t\leq T}\|\Gamma(w)(t)-\Gamma(w')(t)\|_{\cH^{1}(0,1)}\leq C_A^{(4)}(\|w\|_{Y^{1}_{0,T}}+\|w'\|_{Y^{1}_{0,T}}+\|v\|_{Y^{1}_{0,T}})\|w-w'\|_{Y^{1}_{0,T}}
\end{align*}
and similarly, 
\begin{align*}
    \bigg(\int_0^T\|\Gamma(w)(t)-\Gamma(w')(t)\|_{\cH^{1}(0,1)}^2dt\bigg)^\frac{1}{2}&\leq C_A^{(5)}(\|w\|_{Y^{1}_{0,T}}+\|w'\|_{Y^{1}_{0,T}}+\|v\|_{Y^{1}_{0,T}})\|w-w'\|_{Y^{1}_{0,T}}\\
    &\leq C_A^{(5)} (2M'+\frac{M'}{K})\|w-w'\|_{Y^{1}_{0,T}}.
\end{align*}

We can apply Banach Fixed Point Theorem as long as $C_A^{(5)}(2M'+\frac{M'}{K})<1$ as well, that is $M'<\frac{1}{C_A^{(5)}}\frac{K}{K+2}.$ That is, we can take $K=1$, and if
$$\big(\|\phi\|_{\cH^1(0,1)}^2+\|f\|_{L^2([0,T];\cH^{-1}(0,1))}^2\big)^{\frac12}<\min\{ \frac{1}{2C'''_A}, \frac{1}{2C^{(5)}_A}\},$$
 we have a $w\in Y^{1}_{0,T,M'}$ such that $\Gamma(w)=w$ and 
$$\|w\|_{Y^{1}_{0,T}}\leq (\|\phi\|_{\cH^1(0,1)}^2+\|f\|^2_{L^2([0,T];\cH^{-1}(0,1))})^{\frac12}.$$

The proof of this proposition is now complete. 
\end{proof}

We shall also mention an estimate of the derivative of $w$ with respect to $t$.
\begin{prop} \label{Prop1}
    Let  $\phi \in \cH^1(0,1)$ and $f\in L^2([0,T];\cH^{-1}(0,1))$. The solution $w$ in Proposition \ref{existsub2} satisfies
    $$\|w_t\|_{Y^{1}_{0,T}}\leq C'(\|\phi\|^2_{\cH^1(0,1)}+\|f\|_{L^2([0,T];\cH^{-1}(0,1))}^2)^{\frac{1}{2}}$$
    for some $C'>0$ provided that $\|\phi\|^2_{\cH^1(0,1)}+\|f\|_{L^2([0,T];\cH^{-1}(0,1))}^2<c^{(1)}<1.$ We also have  $$\|w_t\|_{Y^{2}_{0,T}}\leq C''(\|\phi\|^2_{\cH^2(0,1)}+\|f\|_{L^2([0,T];L^2(0,1))}^2)^{\frac{1}{2}}$$
    if we assume that $\phi\in \cH^2(0,1)$, $f\in L^2([0,T];L^2(0,1))$, and  $(\|\phi\|^2_{\cH^2(0,1)}+\|f\|_{L^2([0,T];L^2(0,1))}^2)^{\frac{1}{2}}<c^{(2)}<1$.
\end{prop}
\begin{proof} With loss of generality such as $\ell = 2$, we only prove the estimate of $\ell=1$. \\

    Since $$w_t(t)=-(I-\Delta)^{-1}[w(t)w_x(t)+(v(t)w(t))_x+v(t)v_x(t)]+Aw(t),$$
    we have 
    \begin{align*}
        \|w_t(t)\|^2_{\cH^{1}(0,1)}&\leq \|-(I-\Delta)^{-1}[w(t)w_x(t)+(v(t)w(t))_x+v(t)v_x(t)]+Aw(t)\|^2_{\cH^{1}(0,1)}\\
        &\leq 16\bigg(\|w(t)w_x(t)\|^2_{\cH^{-1}}+\|(v(t)w(t))_x\|^2_{\cH^{-1}(0,1)}+\|v(t)v_x(t)\|^2_{\cH^{-1}(0,1)}+\|w(t)\|_{\cH^{1}(0,1)}^2\bigg).
    \end{align*}
    
    Note that using $\|uv\|_{\cH^{s}}\leq C\|u\|_{\cH^{s_1}}\|v\|_{\cH^{s_2}}$ provided that $0\leq s\leq \min\{s_1,s_2\}$ and $s_1+s_2>s+\frac{1}{2}$ with $s=0$, $s_1=1$ and $s_2=0$, we have
    $$\|w(t)w_x(t)\|_{\cH^{-1}(0,1)}^2\leq \frac{1}{\pi}\|w(t)w_x(t)\|_{L^2(0,1)}^2\leq C \|w(t)\|_{\cH^{1}(0,1)}^2\|w_x(t)\|_{L^2(0,1)}^2$$
    and similar estimates are true for $\|(v(t)w(t))_x\|^2_{\cH^{-1}(0,1)}$ and $\|v(t)v_x(t)\|^2_{\cH^{-1}(0,1)}$. Therefore, we can conclude that 
    \begin{align} \label{wtestimate}
        \|w_t\|^2_{Y^{1}_{0,T}}\leq C'(\|w\|^4_{Y^{1}_{0,T}}+\|v\|^2_{Y^{1}_{0,T}}\|w\|^2_{Y^{1}_{0,T}}+\|v\|^4_{Y^{1}_{0,T}})+16\|w\|_{Y^{1}_{0,T}}^2
    \end{align}
    and thus (since we assumed $\|\phi\|^2_{\cH^1(0,1)}+\|f\|_{L^2([0,T];\cH^{-1}(0,1))}^2<1$),
    $$\|w_t\|_{Y^{1}_{0,T}}\leq C'(\|\phi\|^2_{\cH^1(0,1)}+\|f\|_{L^2([0,T];\cH^{-1}(0,1))}^2)^{\frac{1}{2}}. $$
\end{proof}
We now consider the estimtate on $Y^{1}_{\tau,T}$.  For $t\geq \tau$
$$v(t)=e^{A(t-\tau)}v(\tau)+\int_{\tau}^{t} e^{A(t-s)}(I-\Delta)^{-1}f(s)ds$$
and 
$$w(t)=e^{A(t-\tau)}w(\tau)-\int_{\tau}^{t}e^{A(t-s)}(I-\Delta)^{-1}[w(s)w_x(s)+(v(s)w(s))_x+v(s)v_x(s)]ds.$$

\begin{prop}\label{Prop2}
If $f\in L^{2}([0,T+\tau];\cH^{-1}(0,1))$, then 
 $$\|v\|_{Y^{1}_{\tau,T}}\leq C''(\|\phi\|_{\cH^{1}(0,1)}+\|f\|_{L^2([0,\tau+T];\cH^{-1}(0,1))}).$$

If $f\in L^{2}([0,T+\tau];L^2(0,1))$, then 
 $$\|v\|_{Y^{2}_{\tau,T}}\leq C'''(\|\phi\|_{\cH^{2}(0,1)}+\|f\|_{L^2([0,\tau+T];L^2(0,1))}).$$

\end{prop}
\begin{proof}
    We first note that 
    $$v(\tau)=e^{A\tau}\phi+\int_0^{\tau}e^{A(\tau-s)}f(s)ds.$$
    Therefore, by Proposition \ref{existsub2}, we have
    \begin{align*}
        \|v(\tau)\|_{\cH^{1}(0,1)}\leq C(\|\phi\|_{\cH^{1}(0,1)}+\|f\|_{L^2([0,\tau];\cH^{-1}(0,1))}).
    \end{align*}
    Then, 
    $$\|v(t)\|_{\cH^{1}(0,1)}^2 \leq 2(e^{-c(t-\tau)}\|v(\tau)\|_{\cH^{1}(0,1)}^2+\int_{\tau}^{t}e^{-c(t-s)}\|f(s)\|_{\cH^{-1}(0,1)}^2ds)$$
    
    whence 
    $$\|v\|^2_{Y^{1}_{\tau,T}}\leq C' 2(\|\phi\|_{\cH^{1}(0,1)}^2+\|f\|_{L^2([0,\tau];\cH^{-1}(0,1))}^2)+\int_{\tau}^{\tau+T}\|f(s)\|_{\cH^{-1}(0,1)}^2ds.$$
    That is,
    $$\|v\|_{Y^{1}_{\tau,T}}\leq C''(\|\phi\|_{\cH^{1}(0,1)}+\|f\|_{L^2([0,\tau+T];\cH^{-1}(0,1))}).$$
\end{proof}
\begin{remark} \label{rmkforfinLinfty}
Indeed, we have the following estimate for $v(t)$.
\begin{align*}
    \|v(\tau)\|_{\cH^{1}(0,1)}
    &\leq e^{-c\tau}\|\phi\|_{\cH^{1}(0,1)}+\frac{1}{\sqrt{2c}} \bigg(\int_0^{\tau}e^{-2c(\tau-s)}\|f(s)\|^2_{\cH^{-1}(0,1)}ds\bigg)^{\frac12}.
\end{align*}

    Moreover, if we assume that $f\in L^{\infty}([0,\infty); \cH^{-1}(0,1))$ (in particular, if $f$ is perodic), then we have
    \begin{align*}
        \|v(\tau)\|_{\cH^{1}(0,1)}\leq e^{-c\tau}\|\phi\|_{\cH^{1}(0,1)}+\frac{1}{\sqrt{c}}\|f\|_{L^{\infty}([0,\infty);\cH^{-1}(0,1))},
    \end{align*}
    and
    \begin{align}\label{estimateA}
        \|v\|_{Y^{1}_{\tau,T}}\leq C''_T (\|\phi\|_{\cH^{1}(0,1)}+\|f\|_{L^{\infty}([0,\infty);\cH^{-1}(0,1))}),
    \end{align}
    where $C''_T$ is a constant independent of $\tau$.\\

    Similarly, we have   \begin{align}\label{estimateB}
        \|v\|_{Y^{2}_{\tau,T}}\leq C'''_T (\|\phi\|_{\cH^{2}(0,1)}+\|f\|_{L^{\infty}([0,\infty);L^2(0,1))}),
    \end{align}
    where $C'''_T$ is a constant independent of $\tau$, provided that $\phi\in\cH^2(0,1)$ and $f\in L^{\infty}([0,\infty); L^2(0,1)).$ 
\end{remark}

\begin{prop} \label{wYlT}
If $(\|w(\tau)\|_{\cH^1(0,1)}^2+\|v\|_{Y^{1}_{\tau,T}}^2)^{\frac12}$ is sufficiently small, then 
    $$\|w\|_{Y^{1}_{\tau,T}}<K(\|w(\tau)\|_{\cH^1(0,1)}^2+\|v\|_{Y^{1}_{\tau,T}}^2)^{\frac12}$$
for some $K>1$.\\

If $(\|w(\tau)\|_{\cH^2(0,1)}^2+\|v\|_{Y^{2}_{\tau,T}}^2)^{\frac12}$ is sufficiently small, then 
    $$\|w\|_{Y^{2}_{\tau,T}}<K(\|w(\tau)\|_{\cH^2(0,1)}^2+\|v\|_{Y^{2}_{\tau,T}}^2)^{\frac12}$$
for some $K>1$.
\end{prop}
\begin{proof}
    Define $\ds \Gamma(q)(t):=e^{A(t-\tau)}w(\tau)-\int_{\tau}^{t}e^{A(t-s)}(I-\Delta)^{-1}(qq_x+(qv)_x+vv_x)(s)ds$ for $t\geq \tau$, where $w$ is the solution to \eqref{eqsub2}.\\

 Consider the sapce $Y^1_{\tau,T,M'}:=\{q\in Y^{1}_{\tau,T}:\|q\|_{Y^{1}_{\tau,T}}\leq M'\}$ with $M'=K(\|w(\tau)\|_{\cH^1(0,1)}^2+\|v\|_{Y^{1}_{\tau,T}}^2)^{\frac12}$. Here $M$ is preserved for the bound of $e^{At}$ on $H^1_0(0,1)$.\\

    We will show the boundedness of $\Gamma$ on $Y^1_{\tau,T,M'}$. Contraction can be shown similarly by following the same argument as in the proof of Proposition \ref{existsub2}. In particular,     
    \begin{align*}
       \sup_{t\in [\tau,T+\tau]} \|\Gamma(q)(t)\|_{\cH^{1}(0,1)}^2&\leq \|w(\tau)\|_{\cH^{1}(0,1)}^2+9\int_\tau^{T+\tau}(\|qq_x\|_{\cH^{-1}(0,1)}^2+\|(qv)_x\|^2_{\cH^{-1}(0,1)}+\|vv_x\|_{\cH^{-1}(0,1)}^2)ds\\
        &\leq \|w(\tau)\|_{\cH^{1}(0,1)}^2+9(\|q\|_{Y^{1}_{\tau,T}}^4+\|q\|_{Y^{1}_{\tau,T}}^2\|v\|^2_{Y^{1}_{\tau,T}}+\|v\|_{Y^{1}_{\tau,T}}^4)\\
        &\leq \|w(\tau)\|_{\cH^{1}(0,1)}^2+9(\frac{3}{2}\|q\|_{Y^{1}_{\tau,T}}^4+\frac{3}{2}\|v\|_{Y^{1}_{\tau,T}}^4)\\
        &\leq \frac{(M')^2}{K^2}+\frac{27}{2}(M')^4+\frac{27}{2K^2}(M')^4,
        \end{align*}

    on the other hand, by Young's convolution inequality,

       \begin{align*}
     &\int_{\tau}^{T+\tau}\|\Gamma(q)(t)\|_{\cH^{1}(0,1)}^2dt\\&\leq \|w(\tau)\|_{\cH^{1}(0,1)}^2+9\int_{\tau}^{T+\tau}\int_\tau^{t}e^{-c(t-s)}(\|qq_x(s)\|_{\cH^{-1}(0,1)}^2+\|(qv)_x(s)\|^2_{\cH^{-1}(0,1)}+\|vv_x(s)\|_{\cH^{-1}(0,1)}^2)dsdt\\
        &\leq  \|w(\tau)\|_{\cH^{1}(0,1)}^2+9C(\|q\|_{Y^{1}_{\tau,T}}^4+\|q\|^2_{Y^{1}_{\tau,T}}\|v\|^2_{Y^{1}_{\tau,T}}+\|v\|_{Y^{1}_{\tau,T}}^4)\\
        &\leq \|w(\tau)\|_{\cH^{1}(0,1)}^2+9C(\frac{3}{2}\|q\|_{Y^{1}_{\tau,T}}^4+\frac{3}{2}\|v\|_{Y^{1}_{\tau,T}}^4)\\
        &\leq \frac{(M')^2}{K^2}+\frac{27C}{2}(M')^4+\frac{27C}{2K^2}(M')^4.
    \end{align*}
    
We choose $K>1$ and $M'$ so that  $ \frac{(M')^2}{K^2}+\frac{27}{2}(M')^4+\frac{27}{2K^2}(M')^4<(M')^2$ and $\frac{(M')^2}{K^2}+\frac{27C}{2}(M')^4+\frac{27C}{2K^2}(M')^4<(M')^2$. Then we see that $\Gamma$ is bounded on $Y^1_{\tau,T,M'}$.\\

 We can conclude that for some $K>1$ and sufficiently small $M'$ there exists a unique $q\in Y^{1}_{\tau,T,M'}$ that $q=\Gamma(q)$ and $$\|q\|_{Y^{1}_{\tau,T}}<K(\|w(\tau)\|_{\cH^1(0,1)}^2+\|v\|_{Y^{1}_{\tau,T}}^2)^{\frac12},$$
 and $q=w$. 
\end{proof}

We need an estimate of $(w, w_t)$ with an upper bound relying on $v$.

\begin{prop}\label{Prop4}
Assume that $\sup_{t>0}\|v\|_{Y^{i}_{t,T}}^2$ ($i=1,2$) is sufficiently small. Then, 
$$\|w\|_{Y^{i}_{\tau,T}}+\|w_t\|_{Y^i_{\tau,T}}\leq C'\sup_{t>0}\|v\|_{Y^{i}_{t,T}}^2$$
for $i=1,2$.
\end{prop}
\begin{proof}
    Consider the equation \eqref{eqsub2} on $[\tau, t+\tau]$ ($\tau,t>0$) is 
    \begin{align*}
        w(t+\tau)&=e^{At}w(\tau)-\int_{\tau}^{t+\tau}e^{A(t+\tau-s)}(I-\Delta)^{-1}[ww_x+(wv)_x](s)ds-\int_{\tau}^{t+\tau}e^{A(t+\tau-s)}(I-\Delta)^{-1}(vv_x)(s)ds \\
        &=:e^{At}w(\tau)+Q(t+\tau,\tau,w(\tau),v)+P(t+\tau,\tau,v).
    \end{align*}
    
    Using Proposition \ref{wYlT}, if $\|w(\tau)\|_{\cH^{1}(0,1)}^2+\|v\|_{Y^{1}_{\tau,T}}^2$ is small, then it suffices to estimate $\|w(\tau)\|_{\cH^{1}(0,1)}$.\\
    
    Let $w_k:=w(kT)$ for $k\in \N$. Then, applying above relationship to $\tau=T$, we have 
    \begin{align*}
        w_k=e^{AT}w_{k-1}+Q(kT,(k-1)T,w_{k-1},v)+P(kT,(k-1)T,v)
    \end{align*}
    with $w_0=w(0)\equiv 0$.\\

    We first provide an estimate when $\tau=(k-1)T$ for some $k\in \N$, and then we will show for the case $(k-1)T<\tau<kT$ for some $k\in \N$. By Lemma \ref{bilinearest}, 
    \begin{align*}
        &\|w_k\|_{\cH^1(0,1)}\\
        &\leq e^{-cT}\|w_{k-1}\|_{\cH^{1}(0,1)}+\int_{(k-1)T}^{kT}e^{-c(kT-s)}\big[\|ww_x\|_{\cH^{-1}(0,1)}+\|(wv)_x\|_{\cH^{-1}(0,1)}+\|vv_x\|_{\cH^{-1}(0,1)}\big]ds\\
        &\leq e^{-cT}\|w_{k-1}\|_{\cH^{1}(0,1)}+\frac{C}{\sqrt{2c}}\Big(\|w\|_{Y^{1}_{(k-1)T,T}}^2+\|w\|_{Y^{1}_{(k-1)T,T}}\|v\|_{Y^{1}_{(k-1)T,T}}+\|v\|_{Y^{1}_{(k-1)T,T}}^2\Big) \\
        &\leq e^{-cT}\|w_{k-1}\|_{\cH^{1}(0,1)}+\frac{C'}{\sqrt{2c}}\|w\|_{Y^{1}_{(k-1)T,T}}^2+\frac{C'}{\sqrt{2c}}\|v\|_{Y^{1}_{(k-1)T,T}}^2 \\
        &\leq  e^{-cT}\|w_{k-1}\|_{\cH^{1}(0,1)}+\frac{C''K^2}{\sqrt{2c}}(\|w_{k-1}\|_{\cH^{1}(0,1)}^2+\|v\|_{Y^{1}_{(k-1)T,T}}^2)+\frac{C'}{\sqrt{2c}}\|v\|_{Y^{1}_{(k-1)T,T}}^2 \\
        &\leq \bigg(e^{-cT}+\frac{C''K^2}{\sqrt{2c}}\|w_{k-1}\|_{\cH^{1}(0,1)}\bigg)\|w_{k-1}\|_{\cH^{1}(0,1)}+\frac{C''K^2+C'}{\sqrt{2c}}\|v\|_{Y^{1}_{(k-1)T,T}}^2
    \end{align*}    
    
If $\|w_{j}\|_{\cH^{1}(0,1)}\leq \eta$ for $j\in \{1, \cdots,k-1\}$ and $\xi:=e^{-cT}+\frac{C''K^2}{\sqrt{2c}}\eta$, then we can rewrite the above inequality as 
$$\|w_{k}\|_{\cH^{1}(0,1)}\leq \xi \|w_{k-1}\|_{\cH^{1}(0,1)}+\bigg(\frac{C''K^2+C'}{2\sqrt{2c}}\bigg)\|v\|_{Y^{1}_{(k-1)T,T}}^2.$$
Then, using iteration as in the proof of \cite{WZ}*{Lemma 3.4}, we have 
$$\|w_k\|_{\cH^{1}(0,1)}\leq\bigg(\frac{C''K^2+C'}{2\sqrt{2c}}\bigg)\frac{1}{1-\xi}\sup_{t>0}\|v\|_{Y^{1}_{t,T}}^2$$
provided that $\eta$ and $\sup_{t>0}\|v\|_{Y^{1}_{t,T}}^2$ are sufficiently small so that $\xi<1$ and small enough to apply Proposition \ref{wYlT}.\\

For $\tau\in((k-1)T,kT)$, we write $\tau=(k-1)T+\Ga T$, where $\Ga\in(0,1)$, and
$$w(\tau)=e^{A(\Ga T)}w((k-1)T)-\int_{(k-1)T}^{(k-1)T+\Ga T}e^{A(\tau-s)}(I-\Delta)^{-1}[ww_x+(wv)_x+vv_x]ds.$$

A similar reasoning gives us 
\begin{align*}
    \|w(\tau)\|_{\cH^{1}(0,1)}&\leq e^{-c\Ga T}\|w_{k-1}\|_{\cH^{1}(0,1)}+\frac{3}{2\sqrt{2c}}\|w\|_{Y^{1}_{(k-1)T,\Ga T}}^2+\frac{3}{2\sqrt{2c}}\|v\|_{Y^{1}_{(k-1)T,\Ga T}}^2\\
    &\leq e^{-c\Ga T}\bigg(3+\frac{3K^2}{2\sqrt{2c}}\bigg)\frac{1}{1-\xi}\|v\|_{Y^{1}_{(k-2)T,T}}^2+\frac{3}{2\sqrt{2c}}\|w\|_{Y^{1}_{(k-1)T,\Ga T}}^2+\frac{3}{2\sqrt{2c}}\|v\|_{Y^{1}_{(k-1)T,\Ga T}}^2 \\
    &\leq e^{-c\Ga T}\bigg(3+\frac{3K^2}{2\sqrt{2c}}\bigg)\frac{1}{1-\xi}\|v\|_{Y^{1}_{(k-2)T,T}}^2+\frac{3}{2\sqrt{2c}}\|w\|_{Y^{1}_{(k-1)T,T}}^2+\frac{3}{2\sqrt{2c}}\|v\|_{Y^{1}_{(k-1)T,T}}^2 \\
        &\leq e^{-c\Ga T}\bigg(3+\frac{3K^2}{2\sqrt{2c}}\bigg)\frac{1}{1-\xi}\|v\|_{Y^{1}_{(k-2)T,T}}^2+\frac{3}{2\sqrt{2c}}(\|w_{k-1}\|^2_{H^{1}(0,1)}+\|v\|_{Y^{1}_{(k-1)T,T}}^2)+\frac{3}{2\sqrt{2c}}\|v\|_{Y^{1}_{(k-1)T,T}}^2 \\
        &\leq C_{\Ga,c,K,\xi}\sup_{t>0}\|v\|_{Y^{1}_{t,T}}^2.
\end{align*}

 Combined with \eqref{wtestimate}, it is inferred that 
 
\begin{align*}
    \|w_t\|_{Y^{1}_{\tau,T}}^2\leq C''\sup_{t>0}\|v\|_{Y^1_{t,T}}^8+16C'\sup_{t>0}\|v\|_{Y^1_{t,T}}^4.
\end{align*}

Since we can always let $\|v\|_{Y^1_{t,T}}$ sufficiently small for $t$: $\sup_{t>0}\|v\|_{Y^1_{t,T}}<1$, we have \begin{align*}
    \|w_t\|_{Y^{1}_{\tau,T}}\leq C'''\sup_{t>0}\|v\|_{Y^1_{t,T}}^2.
\end{align*}
\end{proof}

\subsection{The case $\ell=3$:}
Similar to the proof of Lemmas \ref{aprih3} and \ref{apriforceh3}, we differentiate \eqref{eqsub2} both sides with respect to $x$, then we have 
\begin{align}\label{deqsub2}
    w_{xt}+w_{xx}-w_{xxx}-w_{xxxt}+(ww_x)_x+(vw)_{xx}=-(vv_x)_x.
\end{align}
Therefore, after multiplying $w_{xxx}$ both sides \eqref{deqsub2} and integrating over $x\in (0,1)$, we have to deal with the product of the form $Fw_{xxx}$, where $F(\cdot, t)$ has regularity at most $\cH^2(0,1)$ for fixed $t$. For the term of the form $F w_{xxx}$, we can estimate as follows:
$$\int_0^1 Fw_{xxx}dx\leq \|F\|_{L^2(0,1)}\|w_{xxx}\|_{L^2(0,1)}\leq \frac{4}{\Ge}\|F\|_{L^2(0,1)}^2+\Ge\|w_{xxx}\|^2_{L^2(0,1)}$$
for any $\Ge>0$ and $\Ge$ will be chosen later.\\

Therefore, we have 
\begin{align} \label{wxxx}
    &\frac{d}{dt}\|w_{xxx}\|_{L^2(0,1)}^2+\|w_{xxx}\|_{L^2(0,1)}^2 \nonumber \\
    &\leq \frac{8}{\Ge}(M^2\|w\|_{\cH^2(0,1)}^2+M^2\|w_t\|_{\cH^1(0,1)}^2+\|(ww_x)_x\|_{L^2(0,1)}^2+M^2\|(vw)_{x}\|_{\cH^1(0,1)}^2+\|(vv_x)_{x}\|_{L^2(0,1)}^2)+10\Ge\|w_{xxx}\|_{L^2(0,1)}^2.
\end{align}
We now choose $\Ge<\frac{1}{10}$, and integrate over $t$, then we have 
\begin{align*}
    &\|w_{xxx}(t)\|^2_{L^2(0,1)}-e^{-(1-10\Ge)t}\|\phi_{xxx}\|_{L^2(0,1)}^2\\
    &\leq \frac{8}{\Ge}\int_0^t e^{-(1-10\Ge)(t-s)}(M^2\|w\|_{\cH^2(0,1)}^2+M^2\|w_t\|_{\cH^1(0,1)}^2+\|(ww_x)_x\|_{L^2(0,1)}^2+M^2\|(vw)_{x}\|_{\cH^1(0,1)}^2+\|(vv_x)_x\|_{L^2(0,1)}^2)ds.
\end{align*}
Thus, by the aid of Propositions  \ref{existsub1}-\ref{Prop2}, 
\begin{align*}
    &\|w_{xxx}\|_{L^\infty([0,T];L^2(0,1))}^2+\|w_{xxx}\|_{L^{2}([0,T];L^2(0,1))}^2 \\
    &\leq C_{\Ge,M}\bigg(\|\phi_{xxx}\|^2_{L^2(0,1)}+\|w\|_{Y^{2}_{0,T}}^2+\|w_t\|_{Y^1_{0,T}}^2+\|w\|_{Y^{1}_{0,T}}^2\|w_x\|^2_{Y^1_{0,T}}+\|v\|^2_{Y^1_{0,T}}\|w\|^2_{Y^1_{0,T}}+\|v\|^2_{Y^{1}_{0,T}}\|v_x\|^2_{Y^1_{0,T}}\bigg)\\
    &\leq C'_{\Ge,M}\bigg(\|\phi_{xxx}\|^2_{L^2(0,1)}+\|\phi\|_{\cH^2(0,1)}^2+\|f\|_{L^2([0,T];L^2(0,1))}^2\bigg)=C''_{\Ge,M}(\|\phi\|_{\cH^3(0,1)}^2+\|f\|_{L^2([0,T];L^2(0,1))}^2),
\end{align*}
which implies 
$$\|w\|_{Y^3_{0,T}}^2\leq 2C^{(3)}_{\Ge,M}(\|\phi\|_{\cH^3(0,1)}^2+\|f\|_{L^2([0,T];L^2(0,1))}^2).$$

If we integrate \eqref{wxxx} with respect to time over $[\tau,t]$, we have 
\begin{align*}
    &\|w_{xxx}(t)\|_{L^2(0,1)}^2-e^{-(1-10\Ge)(t-\tau)}\|w_{xxx}(\tau)\|_{L^2(0,1)}^2\\
    &\leq \frac{8}{\Ge}\int_{\tau}^{t} e^{-(1-10\Ge)(t-s)}(M^2\|w\|_{\cH^2(0,1)}^2+M^2\|w_t\|_{\cH^1(0,1)}^2+\|(ww_x)_x\|_{L^2(0,1)}^2+M^2\|(vw)_x\|_{\cH^1(0,1)}^2+\|(vv_x)_x\|_{L^2(0,1)}^2)ds.
\end{align*}
Therefore, we have 
$$\|w\|_{Y^3_{\tau,T}}^2\leq 2C^{(3)}_{\Ge,K,M}(\|w(\tau)\|_{\cH^3(0,1)}^2+\|f\|_{L^2([0,\tau+T];L^2(0,1))}^2).$$

Also, 
\begin{eqnarray*}
\|w_t\|_{Y^3_{\tau,T}}\leq \|w\|_{Y^2_{\tau,T}}+\|w\|_{Y^3_{\tau,T}}+\|w\|_{Y^1_{\tau,T}}.
\end{eqnarray*}

To summarize, we have the following results of $(w, w_t)$ in $Y_{0,T}^3$:

\begin{prop}
    If $\|\phi\|_{\cH^2(0,1)}^2+\|f\|_{L^2([0,T];L^2(0,1))}^2$ is sufficiently small, then there exists $C_{\Ge,M},\ C_{\Ge,K,M}>0$ such that 
    $$\|w\|_{Y^3_{0,T}},\|w_t\|_{Y^3_{0,T}}\leq C_{\Ge,M}(\|\phi\|_{\cH^3(0,1)}^2+\|f\|_{L^2([0,T];L^2(0,1))}^2)^{\frac12}$$ 
    and 
    $$\|w\|_{Y^3_{\tau,T}},\|w_t\|_{Y^3_{\tau,T}}\leq C_{\Ge,K,M}(\|w(\tau)\|_{\cH^3(0,1)}^2+\|f\|_{L^2([0,\tau+T];L^2(0,1))}^2)^{\frac12}.$$
\end{prop}

Our next aim is to prove $\|w\|_{Y^{3}_{\tau,T}}\leq C\sup_{t>0}\|v\|_{Y^3_{t,T}}^2.$ Note that
\begin{align*}
    &\|w_{xxx}(kT)\|_{L^2(0,1)}^2\\
    &\leq \frac{8}{\Ge}\int_{(k-1)T}^{kT}e^{-(1-10\Ge)(kT-s)}(M^2\|w\|_{\cH^2(0,1)}^2+M^2\|w_t\|_{\cH^1(0,1)}^2+\|(ww_x)_x\|_{L^2(0,1)}^2\\
    &\quad\quad\quad+M^2\|(vw)_x\|_{\cH^1(0,1)}^2+\|(vv_x)_x\|_{L^2(0,1)}^2)ds+e^{-(1-10\Ge)T}\|w_{xxx}((k-1)T)\|_{L^2(0,1)}^2\\
    &\leq C_{\Ge,M}(\sup_{t>0}\|v\|_{Y^2_{t,T}}^4+\sup_{t>0}\|v\|_{Y^1_{t,T}}^4)+C_{\Ge,M}'(\|w\|_{Y^2_{(k-1)T,T}}^4+\|v\|_{Y^2_{(k-1)T,T}}^4)+e^{-(1-10\Ge)T}\|w_{xxx}((k-1)T)\|_{L^2(0,1)}^2,
\end{align*}

whence 
\begin{align*}
    \|w_k\|_{H^3(0,1)}&=\|w_k\|_{H^2(0,1)}+\|(w_k)_{xxx}\|_{L^2(0,1)}\\
    &\leq M\|w_k\|_{\cH^2(0,1)}+\|(w_k)_{xxx}\|_{L^2(0,1)}\\
    &\leq C_{K, c,\xi,M}\sup_{t>0}\|v\|_{Y^2_{t,T}}^2+2(C_{\Ge,M}+C'_{\Ge,M})\sup_{t>0}\|v\|_{Y^2_{t,T}}^2+e^{-(1-10\Ge)T/2}\|w_{k-1}\|_{H^3(0,1)}+C_{K,c,\xi,M}'\sup_{t>0}\|v\|_{Y^{2}_{t,T}}^{\textcolor{red}{4}},
\end{align*} 
in which Proposition \ref{Prop4} is used to control $\|w\|^4_{Y^2_{(k-1)T,T}}$.\\

We are now ready to repeat the argument as in the proof of Proposition \ref{Prop4} applied to $H^3(0,1)$. Therefore, we have 
$$\|w_k\|_{H^3(0,1)}\leq C_{K,c,\Ge,M}\sup_{t>0}\|v\|_{Y^2_{t,T}}^2,$$
which implies 
$$\|w_k\|_{\cH^3(0,1)}\leq MC_{K,c,\Ge,M}\sup_{t>0}\|v\|_{Y^2_{t,T}}^2.$$

Thus, we can conclude the following proposition:

\begin{prop}
    If $\sup_{t>0}\|v\|_{Y^2_{t,T}}$ is sufficiently small, then 
    $$\|w\|_{Y^3_{\tau,T}},\|w_t\|_{Y^3_{\tau,T}}\leq C'''_M\sup_{t>0}\|v\|_{Y^2_{t,T}}^2. $$
\end{prop}

\subsection{Proof of  Theorem \ref{th2-1}}
We are in position to establish the existence and uniqueness of the solution to \eqref{eqmain}.
\begin{proof}[Proof of Theorem \ref{th2-1}]\ \\
    Since $u(t)=v(t)+w(t)$, by Proposition \ref{Prop2} we obtain
    \begin{align*}
        \|u\|_{Y^{1}_{\tau,T}}&\leq \|v\|_{Y^{1}_{\tau,T}}+\|w\|_{Y^{1}_{\tau,T}} \\&\leq C_{K,c} (\|\phi\|_{\cH^{1}(0,1)}+\|f\|_{L^2([0,\tau+T];\cH^{-1}(0,1)})\leq C_{K,c}(\|\phi\|_{\cH^{1}(0,1)}+\|f\|_{L^2([0,\infty);\cH^{-1}(0,1))})
    \end{align*}
    provided that $\|\phi\|_{\cH^{1}(0,1)}+\|f\|_{L^2([0,\infty);\cH^{-1}(0,1))}$ is small so that $\sup_{t>0}\|v\|_{Y^{1}_{t,T}}$ is small enough to apply Proposition \ref{Prop4}.  \\

    If $f\in L^{\infty}([0,\infty);\cH^{-1}(0,1))$, then by Remark \ref{rmkforfinLinfty} it suffices to take $\|\phi\|_{\cH^{1}(0,1)}+\|f\|_{L^\infty([0,\infty);\cH^{-1}(0,1))}$ but the smallness of this constant is depending on $T$ and independent of $\tau$.\\

    The estimate of $u$ for $Y^2_{\tau,T}$ and $Y^3_{\tau,T}$ can be obtained following similar fashion inductively. \\

    When we prove under $\ell\in\{1,2,3\}$, the nonlinear interpolation theory can lead to conclusions for all $\ell\in[1, 3]$ if $(\phi,f)\in \cH^{\ell}\times \cH^{\ell-2}(0,1)$. Briefly, for the solution map defined by the (\ref{BBM})-(\ref{IBVP}): $(\phi, f)\mapsto u$, there hold already the $Y^3_{\tau, T}$-estimate of $u$ which was just done by induction, and $Y^1_{\tau, T}$-estimate of $w = u_1-u_2$: $\|w\|_{Y^1_{\tau, T}}\le C_T\|(\phi_1-\phi_2,f_1-f_2)\|_{\cH^{\ell}\times\cH^{\ell-2}}$, in which $(\phi_i,f_i)\mapsto u_i, i=1,2,$ and same calculation follows as in derivation of (\ref{w_l1}) in later section, then $u\in \cH^\ell$. We would point out that this interpolation argument is a version of a close one in Section 4 of \cite{BSZ2} and earlier cited one \cite{BS} by Bona and Scott. \\

    When integer $\ell_0>3$, the corresponding regularity ($\cH^{\ell_0}$ estimates) of solution can be proved by differentiating \eqref{eqmain} with respect to $x$ with $\lfloor \frac{\ell}{2}\rfloor$-many times and passed through a similar argument as $\ell=3$. Interpolation argument works again for all $\ell\in[3, \ell_0].$\\
    
    This completes the proof of the whole theorem.
   
\end{proof} 

\section{Periodic solution and stability}
In this section, we assume that $f$ has time-periodicity: $f(x,t+\theta)=f(x,t)$ for all $t\geq 0$, i.e., $f(\cdot,t)$ has period $\theta$. Note that if $w(x,t):=u(x,t+\Gth) -u(x,t)$ we can see that $w$ solves 
\begin{equation}\label{eqmainperi}
\begin{cases}
    w_t+w_x+\frac{1}{2}([u(x,t+\Gth)+u(x,t)]w)_x-w_{xx}-w_{xxt}=0,  \quad (x, t)\in [0,1]\times [0, \infty),&\\
    w(x,0)=u(x,\theta)-\phi(x) &\\
    w(0,t)=w(1,t)=0.&
    \end{cases}
\end{equation}
Therefore, we first focus on 
\begin{equation}\label{eqmainperi1}
\begin{cases}
    w_t+w_x+(wa)_x-w_{xx}-w_{xxt}=0,  \quad (x, t)\in [0,1]\times [0, \infty),&\\
    w(x,0)=\psi(x)&\\
    w(0,t)=w(1,t)=0.&
    \end{cases}
\end{equation}
for some function $a$, and we can apply to $a(x,t)=\frac{u(x,t+\Gth)+u(x,t)}{2}$. Note that the equation \eqref{eqmainperi1} is a linear PDE. Therefore, the solution can be written as 
$$w(t)=e^{At}\psi(x)-\int_0^te^{A(t-s)}(I-\Delta)^{-1}(aw)_x(s)ds.$$

\begin{lemma}[A priori estimate for \eqref{eqmainperi1}]\label{aprieqap1}\ \\
Let $\ell\in [1,2]$. Suppose $\|a\|_{Y^{\ell}_{0,T}}$ is sufficiently small, then 
    \begin{align}\label{w_l1}
        \|w\|_{Y^{\ell}_{0,T}}\leq C\|\psi\|_{\cH^{\ell}(0,1)}
    \end{align}
    for some $C$ independent of $T$ and $\|a\|_{Y^{\ell}_{0,T}}$.\\

    If $\sup_{t\geq t_0}\|a\|_{Y^{\ell}_{t,T}}$ is small enough, then for any $\tau\geq t_0$
   \begin{align} \label{WYtauTtoWtau}
       \|w\|_{Y^{\ell}_{\tau,T}}\leq C\|w(\tau)\|_{\cH^{\ell}(0,1)}
   \end{align}
   and
      \begin{align} \label{WtYtauTtoWtau}
       \|w_t\|_{Y^{\ell}_{\tau,T}}\leq C'\|w(\tau)\|_{\cH^{\ell}(0,1)}.
   \end{align}
\end{lemma}

\begin{proof}
    First, using the semigroup solution of $w$, we have
    \begin{align*}
        \|w(t)\|_{\cH^{\ell}(0,1)}^2&\leq 4e^{-ct}\|\psi\|_{\cH^{\ell}(0,1)}^2+4\bigg(\int_0^te^{-c(t-s)}\|(aw)_x(s)\|_{\cH^{\ell-2}(0,1)}ds\bigg)^2
    \end{align*}
    Then following the argument as in the proof of Proposition \ref{existsub1}, we have that 
    $$\|w\|_{Y^{\ell}_{0,T}}^2\leq C\|\psi\|_{\cH^\ell(0,1)}^2+C_{c}\|a\|_{Y^{\ell}_{0,T}}^2\|w\|_{Y^{\ell}_{0,T}}^2.$$
    Therefore, if $\|a\|_{Y^{\ell}_{0,T}}^2\leq \frac{1}{C_{c}}$, then 
    $$\|w\|_{Y^{\ell}_{0,T}}^2\leq C'\|\psi\|_{\cH^{\ell}(0,1)}^2,$$
    
    where $C'$ is independent of $a$.\\

    Next, to show the estimate $\|w\|_{Y^{\ell}_{\tau,T}}$, we first note that the solution can be written as 
    $$w(t):=e^{A(t-\tau)}w(\tau)-\int_{\tau}^{t}e^{A(t-s)}(I-\Delta)^{-1}(aw)_x(s)ds.$$
    Then, 
    \begin{align*}
        \|w(t)\|_{\cH^{\ell}(0,1)}\leq e^{-c(t-\tau)}\|w(\tau)\|_{\cH^{\ell}(0,1)}+\int_{\tau}^{t}e^{-c(t-s)}\|(aw)_x(s)\|_{\cH^{\ell-2}(0,1)}ds.
    \end{align*}
    Taking supremum over $t\in [\tau,T+\tau]$, we have 
    \begin{align*}
       \sup_{t\in [\tau,T+\tau]} \|w(t)\|_{\cH^{\ell}(0,1)}&\leq \|w(\tau)\|_{\cH^{\ell}(0,1)}+\frac{1}{\sqrt{2c}}\int_{\tau}^{T+\tau}\|(aw)_x(s)\|_{\cH^{\ell-2}(0,1)}ds\\
       &\leq \|w(\tau)\|_{\cH^{\ell}(0,1)}+C''\|a\|_{Y^{\ell}_{\tau,T}}\|w\|_{Y^{\ell}_{\tau,T}}.
    \end{align*}
    On the other hand, if we take $L^2([\tau,T+\tau])$ norm both sides,
     \begin{align*}
      \|w\|_{L^2([\tau,T+\tau];\cH^{\ell}(0,1))}&\leq  \|e^{-c(t-\tau)}\|_{L^2([\tau,T+\tau])}\|w(\tau)\|_{\cH^{\ell}(0,1)}+\bigg\|\int_{\tau}^{t}e^{-c(t-s)}\|(aw)_x(s)\|_{\cH^{\ell-2}(0,1)}ds\bigg\|_{L^2([\tau,\tau+T])}\\
      &\leq \frac{1}{\sqrt{2c}}\|w(\tau)\|_{\cH^{\ell}(0,1)}+C'''\|a\|_{Y^{\ell}_{\tau,T}}\|w\|_{Y^{\ell}_{\tau,T}}.
    \end{align*}
    Therefore, 
    $$\|w\|_{Y^{\ell}_{\tau,T}}\leq C^{(4)}\big(\|w(\tau)\|_{\cH^{\ell}(0,1)}^2+\|a\|_{Y^{\ell}_{\tau,T}}^2\|w\|^2_{Y^{\ell}_{\tau,T}}\big)^{\frac{1}{2}}.$$
    If $\ds \sup_{t\geq t_0}\|a\|_{Y^{\ell}_{t,T}}$ is small enough, it can be inferred that  $$\|w\|_{Y^{\ell}_{\tau,T}}\leq C^{(5)}\|w(\tau)\|_{\cH^{\ell}(0,1)}$$
holds for all $\tau\geq t_0$.\\

Note that 
$$w_t=(I-\Delta)^{-1}[(aw)_x]+Aw.$$

We have
$$\|w_t(t)\|_{\cH^\ell(0,1)}\leq \|(aw)_x(t)\|_{\cH^{\ell-2}(0,1)}+C\|w(t)\|_{\cH^{\ell}(0,1)}.$$
Now we take $L^{2}([\tau,T])$ and $L^{\infty}([\tau,T])$, we have 
$$\|w_t\|_{Y^{\ell}_{\tau,T}}\leq C'\|a\|_{Y^\ell_{\tau,T}}\|w\|_{Y^{\ell}_{\tau,T}}+C\|w\|_{Y^{\ell}_{\tau,T}}\leq C''\|w\|_{Y^{\ell}_{\tau,T}}\leq C'''\|w(\tau)\|_{\cH^{\ell}(0,1)}$$
provided that $\sup_{t\geq t_0}\|a\|_{Y^{\ell}_{\tau,T}}$ is small enough.
\end{proof}

\begin{lemma} \label{aprieqap1h3}
    If $\sup_{t\geq t_0}\|a\|_{Y^{2}_{t,T}}$ is small enough, then for any $\tau\geq t_0$
   \begin{align} \label{WYtauTtoWtau3}
       \|w\|_{Y^{3}_{\tau,T}}\leq C\|w(\tau)\|_{\cH^{3}(0,1)}.
   \end{align}
\end{lemma}
\begin{proof}
Note that 
$$w_{xxxt}+w_{xxx}=-w_{xt}-w_{xx}-(w_xa+wa_x)_x.$$
Therefore, we have 
\begin{align*}
    \frac{d}{dt}\|w_{xxx}\|_{L^2(0,1)}^2+\|w_{xxx}\|^2_{L^2(0,1)}&\leq 2\int_0^1w_{xxx}(-w_{xt}-w_{xx}-(w_xa+wa_x)_x)dx\\
    \frac{d}{dt}\|w_{xxx}\|_{L^2(0,1)}+(1-8\Ge')\|w_{xxx}\|_{L^2(0,1)}^2&\leq \frac{8}{\Ge'}(\|w_{xt}\|_{L^2(0,1)}^2+M^2\|w\|^2_{\cH^2(0,1)}+\|(w_xa)_x\|_{L^2(0,1)}^2+\|(wa_x)_x\|_{L^2(0,1)}^2),
\end{align*}
and taking $\Ge'<\frac18$ we have
\begin{align*}
    &\|w_{xxx}(t)\|^2_{L^2(0,1)}-e^{-(1-8\Ge')(t-\tau)}\|w_{xxx}(\tau)\|^2_{L^2(0,1)}\\
    &\leq C_{\Ge'}\int_\tau^te^{-(1-8\Ge')(t-s)}(M^2\|w_{t}\|_{\cH^1(0,1)}^2+M^2\|w\|^2_{\cH^2(0,1)}+\|(w_xa)_x\|_{L^2(0,1)}^2+\|(wa_x)_x\|_{L^2(0,1)}^2)ds.
\end{align*}
Assuming $\sup_{t\geq t_0}\|a\|_{Y^{2}_{t,T}}$ is sufficiently small, we have 
\begin{align*}
    \|w\|_{Y^{3}_{\tau,T}}&\leq M (M^2\|w\|_{Y^{2}_{\tau,T}}^2+\|w_{xxx}\|_{Y^0_{\tau,T}}^2)^{\frac{1}{2}}\\
    &\leq C_{\Ge',M}'(\|w\|_{Y^2_{\tau,T}}+\|w_{xxx}(\tau)\|_{L^2(0,1)}+\|w_t\|_{Y^1_{\tau,T}}+\|w\|_{Y^2_{\tau,T}}+\|w\|_{Y^2_{\tau,T}}\|a\|_{Y^1_{\tau,T}}+\|w\|_{Y^1_{\tau,T}}\|a\|_{Y^2_{\tau,T}}) \\
    &\leq C''_{\Ge',M}\|w(\tau)\|_{\cH^3(0,1)}
\end{align*}
with the aid of Lemma \ref{aprieqap1}.
\end{proof}
\subsection{Periodic Solutions and local stability}
\begin{proof}[Proof of Theorem \ref{th2-2}]\ \\
Choose $\|\phi\|_{\cH^{\ell}(0,1)}^2+\|f\|_{L^2([0,T];\cH^{\ell-2}(0,1))}^2$ small enough so that $\sup_{t\geq 0}\|\frac{u(x,t+\theta)+u(x,t)}{2}\|_{Y^{\ell}_{t,T}}$ is small enough to apply Lemma \ref{aprieqap1}. We also take $\psi(x)=u(x,\theta)-\phi(x)$.\\

Note that
$$w(t+\tau)=e^{At}w(\tau)-\int_{0}^{t}e^{A(t-s)}(I-\Delta)^{-1}(aw)_x(s+\tau)ds.$$

We define $w_k=w(kT)$ for $k\in \N$. If $\tau=T$, then 
$$w_k=e^{AT}w_{k-1}-\int_{(k-1)T}^{kT}e^{A(kT-s)}(I-\Delta)^{-1}(aw)_x(s)ds.$$
Observe that for $\ell\in \{1,2\}$,
\begin{align*}
    \|w_k\|_{\cH^{\ell}(0,1)}&\leq e^{-cT}\|w_{k-1}\|_{\cH^{\ell}(0,1)}+\int_{(k-1)T}^{kT}e^{-c(kT-s)}\|(aw)_x(s)\|_{\cH^{\ell-2}(0,1)}ds\\
    &\leq  e^{-cT}\|w_{k-1}\|_{\cH^{\ell}(0,1)}+\frac{1}{\sqrt{2c}}\|a\|_{Y^{\ell}_{(k-1)T,T}}\|w\|_{Y^{\ell}_{(k-1)T,T}}\\
    &\leq   e^{-cT}\|w_{k-1}\|_{\cH^{\ell}(0,1)}+\frac{C'}{\sqrt{2c}}\|a\|_{Y^{\ell}_{(k-1)T,T}}\|w_{k-1}\|_{\cH^{\ell}(0,1)} \\
    &= \bigg(e^{-cT}+\frac{C'}{\sqrt{2c}}\sup_{t>0}\|a\|_{Y^{\ell}_{t,T}}\bigg)\|w_{k-1}\|_{\cH^{\ell}(0,1)},
\end{align*}

where we have used the estimate \eqref{WYtauTtoWtau}. \\

By taking $\sup_{t>0}\|a\|_{Y^{\ell}_{t,T}}$ to be small enough, we may assume that $e^{-cT}+\frac{C'}{\sqrt{2c}}\sup_{t>0}\|a\|_{Y^{\ell}_{t,T}}=:\mu<1$. Therefore, 
\begin{align*}
    \|w_k\|_{\cH^{\ell}(0,1)}\leq \mu\|w_{k-1}\|_{\cH^{\ell}(0,1)}\leq \mu^2\|w_{k-2}\|_{\cH^{\ell}(0,1)}\leq \cdots\leq \mu^{k-1}\|w_1\|_{\cH^{\ell}(0,1)}\leq \mu^k\|\psi\|_{\cH^{\ell}(0,1)}.
\end{align*}
We see that 
$$\|w(kT)\|_{\cH^{\ell}(0,1)}\leq e^{-kT(\ln(1/\mu)/T)}\|\psi\|_{\cH^{\ell}(0,1)}.$$

If $\tau=kT+t'$, where $t'\in(0,T)$ and $k\in \N$, since 
\begin{align*}
   \|w(\tau)\|_{\cH^{\ell}(0,1)}\leq \|w\|_{Y^{\ell}_{kT,t'}} \leq C'\|w_{k}\|_{\cH^{\ell}(0,1)}&\leq C' e^{-kT(\ln(1/\mu)/T)}\|\psi\|_{\cH^{\ell}(0,1)}\\
   &\leq C'e^{t'(\ln(1/\mu)/T)}e^{-\tau(\ln(1/\mu)/T)}\|\psi\|_{\cH^{\ell}(0,1)}\\
    &\leq C'\frac{1}{\mu}e^{-\tau(\ln(1/\mu)/T)}\|\psi\|_{\cH^{\ell}(0,1)}.
\end{align*}
Note that in the last step we used the fact that $t'\in (0,T)$.
Therefore, for any $t\geq 0$, we have 
$$\|w(t)\|_{\cH^{\ell}(0,1)}\leq C'\frac{1}{\mu}e^{-t(\ln(1/\mu)/T)}\|\psi\|_{\cH^{\ell}(0,1)},$$
and 
taking $L^2([\tau,T+\tau])$ as well as $L^{\infty}([\tau,T+\tau])$ we have 
$$\|w\|_{Y^{\ell}_{\tau,T}}\leq C'\bigg(\frac{\ln(1/\mu)}{T}+1\bigg)\frac{1}{\mu}e^{-\tau(\ln(1/\mu)/T)}\|u(\theta)-\phi\|_{\cH^{\ell}(0,1)}.$$
These two inequalities are also true for $w_t$ with 3 times the original implicit majorizing constants.\\

 For $\ell=3$, following the detail of proving (\ref{WYtauTtoWtau}), we have 
\begin{align*}
    &\|w_{xxx}(kT)\|^2_{L^2(0,1)}-e^{-(1-8\Ge')T}\|w_{xxx}((k-1)T)\|^2_{L^2(0,1)}\\
    &\leq C_{\Ge'}\int_{(k-1)T}^{kT}e^{-(1-8\Ge')(kT-s)}(M^2\|w_{t}\|_{\cH^1(0,1)}^2+M^2\|w\|^2_{\cH^2(0,1)}+\|(w_xa)_x\|_{L^2(0,1)}^2+\|(wa_x)_x\|_{L^2(0,1)}^2)ds\\
    &\leq C''_{\Ge',M}(\|w_t\|_{Y^1_{(k-1)T,T}}^2+\|w\|_{Y^2_{(k-1)T,T}}^2)+C_{\Ge'}\int_{(k-1)T}^{kT}e^{-(1-8\Ge')(kT-s)}(\|(w_xa)_x\|_{L^2(0,1)}^2+\|(wa_x)_x\|_{L^2(0,1)}^2)ds\\
    &\leq 2C''_{\Ge',M}\|w((k-1)T)\|_{\cH^2(0,1)}^2+C'_{\Ge'}\sup_{t>0}\|a\|_{Y^2_{t,T}}^2\|w_{k-1}\|_{\cH^2(0,1)}^2\\
    &\leq 2C''_{\Ge',M}e^{-2(k-1)T(\ln(1/\mu)/T)}\|\psi\|_{\cH^{2}(0,1)}^2+C'_{\Ge'}\sup_{t>0}\|a\|_{Y^2_{t,T}}^2e^{-2(k-1)T(\ln(1/\mu)/T)}\|\psi\|_{\cH^{2}(0,1)}^2
\end{align*}
Therefore, using iteration as in the proof of \cite{WZ}*{Lemma 3.4}, we can conclude that 
\begin{align*}
    &\|(w_k)_{xxx}\|_{L^2(0,1)}^2\\
    &\leq e^{-k(1-8\Ge')T}\|\psi_{xxx}\|^2_{L^2(0,1)} + C_{\Ge'}\bigg[ 2C''_{\Ge',M}e^{-2(k-1)T(\ln(1/\mu)/T)}\|\psi\|_{\cH^{2}(0,1)}^2+C'_{\Ge'}\sup_{t>0}\|a\|_{Y^2_{t,T}}^2e^{-2(k-1)T(\ln(1/\mu)/T)}\|\psi\|_{\cH^{2}(0,1)}^2\bigg].
\end{align*}
Therefore, we have 
\begin{align*}
    \|w_k\|_{\cH^3(0,1)}&\leq \|w_{k}\|_{\cH^2(0,1)}+\|(w_k)_{xxx}\|_{L^2(0,1)}\\
    &\leq  e^{-kT(\ln(1/\mu)/T)}\|\psi\|_{\cH^{2}(0,1)}+e^{-k(1-8\Ge')T}\|\psi_{xxx}\|_{L^2(0,1)}+C_{\Ge',M}e^{-(k-1)T(\ln(1/\mu)/T)}\|\psi\|_{\cH^{2}(0,1)}^2\\
    &\quad\quad +C'_{\Ge'}\sup_{t>0}\|a\|_{Y^2_{t,T}}e^{-(k-1)T(\ln(1/\mu)/T)}\|\psi\|_{\cH^{2}(0,1)} \\
    &\leq C_{\Ge',M}''e^{-c_{T,\mu}(kT)}\|\psi\|_{\cH^3(0,1)},
\end{align*}
where $c_{T,\mu}>0$ is independent of $k$ and $\sup_{t>0}\|a\|_{Y^2_{t,T}}<1$.\\

Following the argument above, we can conclude that 
$$\|w\|_{Y^{3}_{\tau,T}}\leq C_{\mu,T,\Ge',M}e^{-\tau(\ln(1/\mu)/T)}\|u(\theta)-\phi\|_{\cH^{3}(0,1)},$$
where $C_{\mu,T,\Ge',M}$ is independent of $\tau$.
\end{proof}

\begin{proof}[Proof of Theorem \ref{th2-3}]\ \\
Let $u_k(x)=u(x,k\theta)$ for $k\in \N$. By Theorem \ref{th2-1}, we see $u_k\in \cH^{\ell}(0,1)$. We will first show that $\{u_k\}_k$ is Cauchy in $\cH^{\ell}(0,1)$, then we will show that solution with initial condition $\lim_k u_k$ will be periodic.\\

    Let $m,n$ be integers. Then,
    \begin{align*}
        \|u_{n+m}-u_n\|_{\cH^{\ell}(0,1)}&\leq \sum_{i=0}^{m-1}\|u_{n+i+1}-u_{n+i}\|_{\cH^{\ell}(0,1)}\\
        &\leq \sum_{i=0}^{m-1}\|w((n+i)\theta)\|_{\cH^{\ell}(0,1)}\\
        &\leq \sum_{i=0}^{m-1}C_{\mu,T}e^{-(n+i)\theta(\ln(1/\mu)/T)}\|u(\theta)-\phi\|_{\cH^{\ell}(0,1)}\\
        &\leq C_{\mu,T}\|u(\theta)-\phi\|_{\cH^{\ell}(0,1)}\frac{e^{-n\theta(\ln(1/\mu)/T)}}{1-e^{-\theta(\ln(1/\mu)/T)}}\to 0
    \end{align*}
    as $n\to\infty$. Therefore, $\{ u_{k}\}_k$ is a Cauchy sequence in $\cH^{\ell}(0,1)$ and we will denote $\widetilde{\phi}:=\lim_{n\to\infty}u_n$ in $\cH^{\ell}(0,1)$. We can see that $\|\widetilde{\phi}\|_{\cH^\ell(0,1)}\leq C(\|\phi\|_{\cH^\ell(0,1)}+\|f\|_{L^\infty([0,\infty);\cH^{\ell-2}(0,1))})$.\\

    Now suppose $\widetilde{u}(x,t)$ be the solution to \eqref{eqmain} with initial condition $u(x,0)=\widetilde{\phi}$. Then
    \begin{align*}
    \|\widetilde{u}(\cdot,\theta)-\widetilde{\phi}(\cdot)\|_{\cH^{\ell}(0,1)} &=\|\widetilde{u}(\cdot,\theta)-\widetilde{u}(\cdot,0)\|_{\cH^{\ell}(0,1)}\\
        &\leq \|\widetilde{u}(\cdot,\theta)-u_{n+1}\|_{\cH^{\ell}(0,1)}+\|u_{n+1}-u_n\|_{\cH^{\ell}(0,1)}+\|u_n-\widetilde{u}(\cdot,0)\|_{\cH^{\ell}(0,1)}\\
        &\to 0
    \end{align*}
    by passing $n\to \infty$ by the fact that $\{u_k\}_k$ is Cauchy and $\widetilde{\phi}:=\lim_{n\to\infty}u_n$ in $\cH^{\ell}(0,1)$. To see the first term is small, note that by the mild solution of $\widetilde{u}$ and $u$ with $a(s)=\frac{\widetilde{u}(s)+u(n\theta+s)}{2}$, 
    \begin{align*}
        \widetilde{u}(\cdot,\theta)-u_{n+1}&= e^{A\theta}(\widetilde{\phi}-u_n)-\int_{0}^{\theta}e^{A(\theta-s)}(I-\Delta)^{-1}(a(s)[\widetilde{u}(s)-u(n\theta+s)])_xds
    \end{align*}
    and 
    \begin{align*}
        \|\widetilde{u}(\cdot,\theta)-u_{n+1}\|_{\cH^{\ell}(0,1)}&\leq  C'e^{-c\theta}\|\widetilde{\phi}-u_n\|_{\cH^{\ell}(0,1)}
    \end{align*}
    provided that $\sup_{t>0}\|u\|_{Y^{\ell}_{t,\Gth}}+\sup_{t>0}\|\widetilde{u}\|_{Y^{\ell}_{t,\Gth}}$ small enough by Lemma \ref{aprieqap1} if $\ell\in [1,2]$. For $\ell=3$, we can apply Lemma \ref{aprieqap1h3} to $w=\widetilde{u}-u(n\Gth+\cdot)$, $a=\frac{\widetilde{u}+u(n\Gth+\cdot)}{2}$, and $\psi=\phi-u(n\Gth)$ on $[0,\Gth]$ provided that $\sup_{t>0}\|u\|_{Y^{3}_{t,\Gth}}+\sup_{t>0}\|\widetilde{u}\|_{Y^{3}_{t,\Gth}}$ is small enough.\\

    We now show the local stability of $\widetilde{u}$. Consider $w(x,t)=u(x,t)-\widetilde{u}(x,t)$ and $a(x,t)=\frac{1}{2}(u(x,t)+\widetilde{u}(x,t))$. Then we see that $w$ solves \eqref{eqmainperi1} with $\psi(x)=\phi(x)-\widetilde{\phi}(x)$. By Theorem \ref{th2-1}, we can take both $\sup_{t>0}\|u\|_{Y^{\ell}_{t,T}}+\sup_{t>0}\|\widetilde{u}\|_{Y^{\ell}_{t,T}}$ small enough to apply Theorem \ref{th2-2}. Therefore, we have 
    $$\|u-\widetilde{u}\|_{Y^{\ell}_{\tau,T}}\leq C'_{\mu,T}e^{-\tau(\ln(1/\mu)/T)}\|\psi\|_{\cH^{\ell}(0,1)},$$
    which is exponential convergence.  
\end{proof}
\subsection{Global Stability}
This subsection is dedicated to the proof of Theorem \ref{th2-4}. To prove the global stability, we will first establish the case $\ell=1$ and then we will establish a method to reduce the higher order cases $\ell$ to lower order. Different from \cite{WZ}, differentiating $t$ would not give us higher order regularity, we need to estimate the norm of $u$ in $\cH^{\ell}(0,1)$ directly.\\

In order to prove the global exponential stability for $\cH^{\ell}(0,1)$, it suffices to show the global absorbing property

$$\|u\|_{\cH^{\ell}(0,1)}\leq e^{-ct}\|\phi\|_{\cH^{\ell}(0,1)}+C\delta,$$

provided that $\sup_{t}\|f\|_{\cH^{\ell-2}(0,1)}\leq \delta$. Here we will always assume that $\delta <1$ for simplicity. We will explain how to obtain the global stability after showing the global absorbing property for $\ell=1$ and similar argument will work for any $\ell$. \\

When $\ell=1$, note that $\int_0^1u^2u_xdx=0$, and we have 
\begin{align*}
    \frac{1}{2}\frac{d}{dt}(\|u(t)\|^2_{L^2(0,1)}+\|u_x(t)\|^2_{L^2(0,1)})+\|u_x(t)\|^2_{L^2(0,1)}&\leq \frac{C_M}{b}\|f(t)\|_{\cH^{-1}(0,1)}^2+\frac{b}{4M^2}\|u(t)\|_{\cH^1(0,1)}^2\\
    \frac{1}{2}\frac{d}{dt}(\|u(t)\|^2_{L^2(0,1)}+\|u_x(t)\|^2_{L^2(0,1)})&\leq \frac{C_M}{b}\|f(t)\|_{\cH^{-1}(0,1)}^2+\frac{b}{4}(\|u(t)\|_{L^2(0,1)}^2+\|u_x(t)\|_{L^2(0,1)}^2)\\
    &\quad\quad -\frac{1}{2}\|u_x(t)\|^2_{L^2(0,1)}-\frac{(c')^2}{2}\|u(t)\|_{L^2(0,1)}^2\\
    &\leq \frac{C_M}{b}\|f(t)\|_{\cH^{-1}(0,1)}^2-c'''_b(\|u(t)\|_{L^2(0,1)}^2+\|u_x(t)\|_{L^2(0,1)}^2)
\end{align*}
here we take $\frac{b}{2}<\min\{1,(c')^2\}$. Therefore, 
\begin{align*}
    \|u(t)\|_{\cH^1(0,1)}^2\leq C'_Me^{-2c'''_bt}\|\phi\|_{\cH^1(0,1)}^2+C_{M,b}\sup_{t\geq 0}\|f(t)\|_{\cH^{-1}(0,1)}^2.
\end{align*}

Therefore, if $\delta>0$ is small such that there exists $t_0>0$ that $\|u(t_0)\|_{\cH^1(0,1)}^2$ is small enough such that $\ds \sup_{t\geq t_0}\|u\|_{Y^{\ell}_{t,T}}+\sup_{t\geq t_0}\|\widetilde{u}\|_{Y^{\ell}_{t,T}}$ is small enough to apply Theorem \ref{th2-3}. More precisely, for given $\phi \in \cH^1(0,1)$ and sufficiently small $\ds\sup_{t\geq 0}\|f(t)\|_{\cH^{-1}(0,1)}^2$, then there exists $t_0$ such that $\|u(t_0)\|_{\cH^{1}(0,1)}$ is small, then we apply Theorem \ref{th2-3} to $u(t_0+t)$ and $u-\widetilde{u}\to 0$ in $Y^{\ell}_{\tau,T}$ as $\tau\to\infty$.\\

We can also see that 
\begin{align*}
    \|u_t(t)\|_{\cH^{1}(0,1)}&\leq \|u_x(t)\|_{\cH^{-1}(0,1)}+\|u(t)\|_{\cH^1(0,1)}+\|u(t)\|_{\cH^{-1}(0,1)}+\|u(t)u_x(t)\|_{\cH^{-1}(0,1)}+\|f(\cdot,t)\|_{\cH^{-1}(0,1)} \\
    &\leq C\|u(t)\|_{\cH^1(0,1)}+M\|u(t)\|_{\cH^1(0,1)}^2+\sup_{t\geq 0}\|f(\cdot,t)\|_{\cH^{-1}(0,1)}\\
    &\leq C\sqrt{C'_{M}}e^{-c_b'''(t)}\|\phi\|_{\cH^{1}(0,1)}+2C'_{M}e^{-2c_b'''(t)}\|\phi\|^2_{\cH^{1}(0,1)}\\
    &\quad\quad\quad+(\sqrt{C_{M,b}}+1+C_{M,b}\sup_{t\geq 0}\|f(\cdot,t)\|_{\cH^{-1}(0,1)})\sup_{t\geq 0}\|f(\cdot,t)\|_{\cH^{-1}(0,1)}.
\end{align*}

For $\ell>1$, note that
\begin{align*}
    &(I-\Delta)u_t+(I-\Delta)u=-u_x+u+\frac{1}{2}[u^2]_x+f(t) \\
   \implies &(I-\Delta)^{\ell/2}u_t+(I-\Delta)^{\ell/2}u=-(I-\Delta)^{\ell/2-1}u_x+(I-\Delta)^{\ell/2-1}u+\frac{1}{2}(I-\Delta)^{\ell/2-1}[u^2]_x+(I-\Delta)^{\ell/2-1}f(t) \\
   \implies &\frac{d}{dt}\|u\|_{\cH^{\ell}}^2+\|u\|_{\cH^\ell}^2\leq \frac{1}{2\Ge}(\|u\|_{\cH^{\ell-1}}^2+\|u\|_{\cH^{\ell-2}}^2+\frac{1}{2}\|u^2\|_{\cH^{\ell-1}}^2+\|f(t)\|^2_{\cH^{\ell-2}})+8\Ge\|u\|_{\cH^\ell(0,1)}^2.
\end{align*}

For $1<\ell\leq 2$, multiply $e^{(1-8\Ge) t}$ to both sides and integrate over $[0, t]$,

\begin{align*}
    &e^{(1-8\Ge)t}\|u(t)\|_{\cH^{\ell}}^2-\|\phi\|_{\cH^{\ell}}^2\\
    &\leq \frac{1}{2\Ge}\int_0^te^{(1-8\Ge)s}(\|u(s)\|_{\cH^{\ell-1}}^2+\|u(s)\|_{\cH^{\ell-2}}^2+\frac{1}{2}\|u^2(s)\|_{\cH^{1}}^2+\|f(s)\|^2_{\cH^{\ell-2}})ds\\
    &\leq \frac{1}{2\Ge}\int_0^te^{(1-8\Ge)s}(2C'_Me^{-2c_b'''s}\|\phi\|_{\cH^1(0,1)}^2+2C_{M,b}\sup_{s\geq0}\|f(s)\|_{\cH^{-1}}^2+\frac{C}{2}\|u(s)\|_{\cH^{1}}^4)ds+ \frac{e^{(1-8\Ge)t}}{1-8\Ge}\sup_{s\geq 0}\|f(s)\|^2_{\cH^{\ell-2}}\\
     &\leq \frac{1}{2\Ge}e^{(1-8\Ge-2c_b''')t}(C'_M\|\phi\|_{\cH^1(0,1)}^2)+\frac{1}{2\Ge} e^{(1-8\Ge-4c_b''')t}(C'_Me^{t}\|\phi\|_{\cH^1(0,1)}^2)^2+e^{(1-8\Ge)t}C'_{M,b,\Ge}\sup_{s\geq 0}\|f(s)\|^2_{\cH^{\ell-2}}
\end{align*}
\begin{align*}
    \|u(t)\|_{\cH^{\ell}}^2
     &\leq Ce^{-(1-8\Ge)t}\|\phi\|_{\cH^{\ell}}^2+C'_{M,\Ge}e^{-2c_b'''t}(\|\phi\|_{\cH^{\ell}}^2+\|\phi\|_{\cH^{\ell}}^4)+C'_{M,b,\Ge}\sup_{s\geq 0}\|f(s)\|^2_{\cH^{\ell-2}}.
\end{align*}

We have the absorbing property if we choose $\max\{\frac{1-2c_{b}'''}{8},0\}<\Ge<\frac{1}{8}$.\\

For $\ell\geq 2$, note that $\ell-1\geq 1$ so $\cH^{\ell-1}$ is an algebra in $C([0, 1])$, and we have $\|u^2\|_{\cH^{\ell-1}}\leq 
C\|u\|_{\cH^{\ell-1}}^2$. Thus, we have 
\begin{align*}
&\frac{d}{dt}\|u\|_{\cH^{\ell}}^2+\|u\|_{\cH^\ell}^2\leq \frac{1}{2\Ge}(\|u\|_{\cH^{\ell-1}}^2+\|u\|_{\cH^{\ell-2}}^2+\frac{C}{2}\|u\|_{\cH^{\ell-1}}^4+\|f(t)\|^2_{\cH^{\ell-2}})+8\Ge\|u\|_{\cH^\ell(0,1)}^2.
\end{align*}
Then we can apply estimates for $\ell\in(2, 3]$. In detail, $\ell-1\in (1,2]$, for which we have proved the desired inequality. If $\ell\in (3, 4], (4, 5], \cdots,$ result works inductively. Therefore, we can conclude the absorbing property for $\ell\geq 2.$ Global stability results follows as mechanism in that of $\cH^1$. 

\begin{remark}
    Before we completely conclude the discussion on \eqref{eqmain}, we could point out following the proof of Theorem \ref{th2-1} for $\ell\ge 3$, that according to regularity argument (so-called ``bootstrap" argument in parabolic equations), the high regularity of solution can be obtained by differentiating \eqref{eqmain} with respect to x $\lfloor \frac{\ell}{2}\rfloor$-many times for $\ell\geq 3$. For instance, \cite{Qin} by Qin used this argument and derived the detained work in high space $H^\ell, \ell = 4$ for 1D hyperbolic-parabolic coupled system from Navier-Stokes flow. 
\end{remark}
\section{A pseudo-parabolic equation on $[0,1]$}

In this section, we shall focus on another homogeneous two-point boundary problem (\ref{eqmain2}):

\begin{equation}\label{whole}
\begin{cases}
    u_t+u_x-u_{xx}-u_{xxt}+[F(u)]_x=[\Phi(u_x)]_x+(I-\Delta)[G(u)]+f(x,t),  \quad (x, t)\in (0,1)\times [0, \infty),&\\
    u(x,0)=\phi(x) &\\
    u(0,t)=u(1,t)=0.&
    \end{cases}
\end{equation}
Here, $F$, $\Phi$ are $G$ are operators which are not necessarily linear. 
More specifically, we write $F\in \Lip_{loc}(\Ga,\Gb;C_F,N_F)$ if $F:\R\mapsto\R$ satisfies
\begin{align*}
&F(0)=0,\\
    &\|F(u)-F(v)\|_{\cH^{\Gb}(0,1)}\leq C_F\|u-v\|_{\cH^{\Ga}(0,1)}(\|u\|_{\cH^{\Ga}(0,1)}^{N_F-1}+\|v\|_{\cH^{\Ga}(0,1)}^{N_F-1})
\end{align*}
for all $u,v$ with $\|u\|_{\cH^{\Ga}(0,1)}, \|v\|_{\cH^{\Ga}(0,1)}\leq 1$. \\

We also write $F\in \Lip(\Ga,\Gb;C_F,N_F)$ if $F:\R\mapsto\R$ satisfies
\begin{align*}
&\|F(u)\|_{\cH^{\Gb}(0,1)}\leq C_F\|u\|_{\cH^{\Ga}(0,1)}(1+\|u\|_{\cH^{\Ga}(0,1)}^{N_{F}})\\
    &\|F(u)-F(v)\|_{\cH^{\Gb}(0,1)}\leq C_F\|u-v\|_{\cH^{\Ga}(0,1)}(\|u\|_{\cH^{\Ga}(0,1)}^{N_F}+\|v\|_{\cH^{\Ga}(0,1)}^{N_F})
\end{align*}
for all $u,v$.

\begin{eg}
Let $\ell\geq 1$. First recall that $(u,v)\mapsto uv$ is a bounded mapping from $\cH^{\ell}(0,1)\times \cH^{\ell}(0,1)$ to $\cH^{\ell-1}(0,1)$ provided that $\ell>\frac{1}{2}$, so inductively, $\prod_{k=1}^{K}u_k\in \cH^{\ell-1}(0,1)$ provided that $u_k\in \cH^{\ell}(0,1)$ if $\ell>-\frac{1}{2(K-1)}$, which is always true when $\ell\geq 1$. Moreover, $\cH^{\ell}(0,1)$ forms an algebra if $\ell>\frac{1}{2}$.

Therefore, if $\ell>1+\frac{1}{4}$, then $\Phi(u)=u^3$ will satisfy
\begin{align*}
    \|u^3\|_{\cH^{\ell-1}(0,1)}&\leq C\|u\|_{\cH^{\ell-1}(0,1)}^3,\\
    \|u^3-v^3\|_{\cH^{\ell-1}(0,1)}&\leq C\|u-v\|_{\cH^{\ell-1}(0,1)}(\|u\|_{\cH^{\ell-1}(0,1)}^2+\|u\|_{\cH^{\ell-1}(0,1)}\|v\|_{\cH^{\ell-1}(0,1)}+\|v\|_{\cH^{\ell-1}(0,1)}^2)\\
    &\leq \frac{3C}{2}\|u-v\|_{\cH^{\ell-1}(0,1)}(\|u\|^2_{\cH^{\ell-1}(0,1)}+\|v\|_{\cH^{\ell-1}(0,1)}^2).
\end{align*}
\end{eg}

Throughout this section we assume that $F\in  \Lip_{loc}(\ell-1,\ell;C_F,N_F)$, $\Phi\in \Lip_{loc}(\ell-1,\ell-1;C_\Phi,N_\Phi)$, and $G\in  \Lip_{loc}(\ell,\ell;C_G,N_G)$. We also denote $\ell_{F,\Phi,G}$ to be the infimum of all $\ell\in \R$ that all the conditions hold. More precisely, we impose the following assumptions.

\begin{enumerate}
    \item[(A)] \hypertarget{(A)} We assume the following.  
    \begin{itemize}
        \item $F\in  \Lip_{loc}(\ell,\ell-1;C_F,N_F)$, $\Phi\in \Lip_{loc}(\ell-1,\ell-1;C_\Phi,N_\Phi)$, and $G\in  \Lip_{loc}(\ell,\ell;C_G,N_G)$ ;
        \item $G\in C^{\lceil\ell\rceil}(\R)$, $F,\Phi\in C^{\lceil\ell\rceil-1}(\R)$,
        \item  $F^{(n)}\in \Lip_{loc}(\ell,\ell-1;C_{F^{(n)}},N_{F}-n)$, $\Phi^{(n)}\in \Lip_{loc}(\ell-1,\ell-1;C_{\Phi'},N_{\Phi}-n)$, and $G^{(n+1)}\in \Lip_{loc}(\ell,\ell;C_{G'},N_{G}-n-1)$ for all $\ell>\ell_{F,\Phi,G}$  and for all natural numbers $1\leq n\leq \lceil\ell\rceil-2$, where $\lceil\ell\rceil=\sup\{n\in\Z: \ell\leq n\}$;
        \item  $\ds \sup_{x\in\R}G'(x)-\inf_{y\in\R}\Phi'(y)<\frac{1}{2}$
    \end{itemize}
    \item[(B)] \hypertarget{(B)} We assume the following:
    \begin{itemize}
        \item $G,\Phi\in C^1(\R)$;
        \item $\ds \sup_{x\in\R}G'(x)<\frac{(c')^2}{4}$, where $c'$ is the constant from Poincar\'e inequality;
        \item $\ds \sup_{x\in\R}G'(x)-\inf_{y\in\R}\Phi'(y)<\frac{1}{2}$.
        \end{itemize}
\end{enumerate}

Here are the analogues of Theorems \ref{th2-1} to \ref{th2-4}. 
\begin{theorem}\label{th6-1}
Let $T, \tau>0$ and $\ell\in [1,\infty)\cap (\ell_{F,\Phi,G},\infty)$. Suppose the assumption \hyperlink{(A)}{(A)} holds.
\begin{enumerate}
    \item If $\phi\in \cH^\ell(0,1)$, $f\in L^2([0,\infty);\cH^{\ell-2}(0,1))$, and $\|\phi\|_{\cH^\ell(0,1)}^2+\|f\|^2_{L^2([0,\infty);\cH^{\ell-2}(0,1))}$ is sufficiently small, then there exists a unique solution $u$ to the equation \eqref{whole} and a constant $C>0$ independent of $T$ and $\tau$ such that 
    $$\|u\|_{Y^{\ell}_{0,T}}\leq C(\|\phi\|_{\cH^\ell(0,1)}^2+\|f\|^2_{L^2([0,\infty);\cH^{\ell-2}(0,1))})^{\frac{1}{2}},$$
    and 
    $$\|u\|_{Y^\ell_{\tau,T}}\leq C(\|\phi\|_{\cH^\ell(0,1)}^2+\|f\|_{L^2([0,\infty);\cH^{\ell-2}(0,1))}^2)^{\frac{1}{2}}.$$
    \item  If $\phi\in \cH^\ell(0,1)$, $f\in L^\infty([0,\infty);\cH^{\ell-2}(0,1))$, and $\|\phi\|_{\cH^\ell(0,1)}^2+\|f\|^2_{L^\infty([0,\infty);\cH^{\ell-2}(0,1))}$ is sufficiently small, then there exists a unique solution $u$ to the equation \eqref{whole} and a constant $C>0$ independent of $\tau$ (but dependent on $T$) such that 
    $$\|u\|_{Y^{\ell}_{0,T}}\leq C(\|\phi\|_{\cH^\ell(0,1)}^2+\|f\|^2_{L^\infty([0,\infty);\cH^{\ell-2}(0,1))})^{\frac{1}{2}},$$
    and 
    $$\|u\|_{Y^\ell_{\tau,T}}\leq C(\|\phi\|_{\cH^\ell(0,1)}^2+\|f\|_{L^\infty([0,\infty);\cH^{\ell-2}(0,1))}^2)^{\frac{1}{2}}.$$
\end{enumerate}
\end{theorem}
\begin{theorem}\label{th6-2}
    Under the assumptions \hyperlink{(A)}{(A)}, Theorem \ref{th6-1} (2), and $f$ has temporal-period $\theta$, then $u(x,t)$ has asymptotic temporal-periodicity.
\end{theorem}
\begin{theorem}\label{th6-3}
        Under the assumptions \hyperlink{(A)}{(A)}, Theorem \ref{th6-1} (2), and $f$ has temporal-period $\theta$, then the equation \eqref{eqnonlinear} has a time-periodic solution in $\cH^{\ell}(0,1)$ provided that $\|\phi\|_{\cH^\ell(0,1)}^2+\|f\|^2_{L^2([0,\infty);\cH^{\ell-2}(0,1))}$ is sufficiently small.
\end{theorem}
For global stability, it is much more complicated. To simplify the discussion, we will focus on $\cH^1(0,1)$ and $\cH^2(0,1)$ only.
\begin{theorem}\label{th6-4}
    \begin{enumerate}
        \item If the assumption \hyperlink{(B)}{(B)} holds, then 
        $u$ satisfies 
        $$\|u\|_{\cH^1(0,1)}\leq C e^{-ct}\|\phi\|_{\cH^1(0,1)}+C'\delta$$
        provided that $\sup_{t\ge 0}\|f(t)\|_{\cH^{-1}(0,1)}\leq \delta$. Therefore, $u$ is globally stable in $\cH^1(0,1)$.
    \item If we assume additionally that $F\in \Lip(1,1;C_F,N_{F}')$, $G'\in \Lip(1,1;C_{G'},N_{G'}')$, and $G\in \Lip(1, 0; C_G,N_{G}')$,  then 
        $u$ satisfies 
        $$\|u\|_{\cH^2(0,1)}\leq C e^{-ct}\|\phi\|_{\cH^2(0,1)}\Psi(\|\phi\|_{\cH^2(0,1)},\delta)+C'\delta$$
    \end{enumerate}
    for some function $\Psi:\R^2\to\R$ provided that $\sup_{t\ge 0}\|f(t)\|_{L^2(0,1)}\leq \delta$. Thus, $u$ is globally stable in $\cH^2(0,1)$.
\end{theorem}
\begin{remark}
    The additional assumptions in Item 2 is not too restrictive. Indeed, $F(u)=u^3\in \Lip(1,1;C,2)$. Moreover, if $G(u)=-u^3$, then we can also see that $G'=-3u^2\in \Lip(1,1,C',1)$ and satisfies assumption  \hyperlink{(B)}{(B)} if we take $\Phi\equiv 0$. 

\end{remark}
\subsection{Local well-posedness}
We consider the nonlinear equation
\begin{equation}\label{eqnonlinear}
\begin{cases}
    w_t+w_x-w_{xx}-w_{xxt}+[F(v+w)]_x=[\Phi(v_x+w_x)]_x+(I-\Delta)[G(v+w)],  \quad (x, t)\in [0,1]\times [0, \infty),&\\
    w(x,0)=0&\\
   w(0,t)=w(1,t)=0,&
    \end{cases}
\end{equation}
where $v(x,t)$ is the solution to \eqref{eqsub1}.\\

Note that, formally, the solution to \eqref{eqnonlinear} is
$$w(t)=\int_0^te^{A(t-s)}(I-\Delta)^{-1}\bigg[[-F(v+w)]_x(s)+[\Phi(v_x+w_x)]_x(s)+(I-\Delta)[G(v+w)](s)\bigg]ds.$$

Moreover, 
\begin{align*}
    w_t(t)&=-(I-\Delta)^{-1}w_x(t)-w(t)+(I-\Delta)^{-1}w(t)-(I-\Delta)^{-1}[F(v+w)]_x(t) \\
    &\quad\quad +(I-\Delta)^{-1}[\Phi(v_x+w_x)]_x(t)+G(v+w)(t).
\end{align*}

\begin{theorem} \label{main61}
    Let $\ell_{F,\Phi,G}<2$ and $\ell\in [1,2]\cap (\ell_{F,\Phi,G},2]$. For $T>0$, if $(\|\phi\|_{\cH^\ell(0,1)}^2+\|f\|_{L^2([0,T];\cH^{\ell-2}(0,1))}^2)^{\frac12}$ is small enough and $N_{\min}:=\min\{N_F,N_{\Phi},N_{G}\}\geq 3$, there exists $T>0$ and a unique $w\in Y^{\ell}_{0,T}$ to the equation \eqref{eqnonlinear} such that 
    $$\|w\|_{Y^{\ell}_{0,T}}\leq (\|\phi\|_{\cH^\ell(0,1)}^2+\|f\|_{L^2([0,T];\cH^{\ell-2}(0,1))}^2)^{\frac12}<1$$
    and 
      $$\|w_t\|_{Y^{\ell}_{0,T}}\leq C(\|\phi\|_{\cH^\ell(0,1)}^2+\|f\|_{L^2([0,T];\cH^{\ell-2}(0,1))}^2)^{\frac12}<1$$
\end{theorem}
\begin{proof}
       Let 
    $$\Gamma(w)(t):=\int_0^te^{A(t-s)}(I-\Delta)^{-1}\bigg[[-F(v+w)]_x(s)+[\Phi(v_x+w_x)]_x(s)+(I-\Delta)[G(v+w)](s)\bigg]ds$$
    and 
     $Y^\ell_{0,T,M'}:=\{w\in Y^{\ell}_{0,T}:\|w\|_{Y^{\ell}_{0,T}}\leq M'\}$ with $M'=K(\|\phi\|_{\cH^\ell(0,1)}^2+\|f\|_{L^2([0,T];\cH^{\ell-2}(0,1))}^2)^{\frac12}$. We will choose $T$, $K$, and $M'$ later so that Banach contraction mapping can be applied. But we will assume that $M'<1$.\\

          Let $w\in Y^\ell_{0,T,M'}$. We will consider term by term.\\

     We first focus on the first term. Note that
     \begin{align*}
         \bigg\|\int_0^te^{A(t-s)}(I-\Delta)^{-1}[-F(v+w)]_x(s)ds\bigg\|_{\cH^{\ell}(0,1)}&\leq \int_0^t e^{-c(t-s)}\|[F(v+w)]_x(s)\|_{\cH^{\ell-2}(0,1)}ds \\
         &\leq  \int_0^t e^{-c(t-s)}\|F(v(s)+w(s))\|_{\cH^{\ell-1}(0,1)}ds\\
         &\leq C_{F}2^{N_F} \int_0^t e^{-c(t-s)}(\|v(s)\|_{\cH^{\ell}(0,1)}^{N_F}+\|w(s)\|_{\cH^{\ell}(0,1)}^{N_F})ds.
     \end{align*}
     Then, using the fact that $\|v(s)\|_{\cH^{\ell}(0,1)}\leq \|v\|_{Y^{\ell}_{0,T}}\leq C'(\|\phi\|_{\cH^{\ell}(0,1)}^2+\|f\|_{L^2([0,T];\cH^{\ell-2}(0,1))}^2)^{\frac12}$ from Theorem \ref{existsub1} and 
     $\|w(s)\|_{\cH^{\ell}(0,1)}\leq \|w\|_{Y^{\ell}_{0,T}}\leq M'$ for all $s\in[0,T]$, we have 
     \begin{align*}
         \sup_{t\in[0,T]}\bigg\|\int_0^te^{A(t-s)}(I-\Delta)^{-1}[-F(v+w)]_x(s)ds\bigg\|_{\cH^{\ell}(0,1)}  
         \leq \frac{2^{N_F}C_F(C')^{N_F}}{K^{N_F}c}(M')^{N_F}+\frac{2^{N_F}C_F}{c}(M')^{N_F}
     \end{align*}
     and 
     \begin{align*}
        \bigg(\int_0^T\bigg\|\int_0^te^{A(t-s)}(I-\Delta)^{-1}[-F(v+w)]_x(s)ds\bigg\|_{\cH^{\ell}(0,1)}^2dt\bigg)^{\frac{1}{2}}         &\leq \frac{2^{N_F}C_F(C')^{N_F}}{K^{N_F}\sqrt{c}}(M')^{N_F}+\frac{2^{N_F}C_F}{\sqrt{2c}}(M')^{N_F}.
     \end{align*}
     
     We now consider the second term. Note that 
      \begin{align*}
         \bigg\|\int_0^te^{A(t-s)}(I-\Delta)^{-1}[\Phi(v_x+w_x)]_x(s)ds\bigg\|_{\cH^{\ell}(0,1)}&\leq \int_0^t e^{-c(t-s)}\|[\Phi(v_x+w_x)]_x(s)\|_{\cH^{\ell-2}(0,1)}ds \\
         &\leq  \int_0^t e^{-c(t-s)}\|\Phi(v_x(s)+w_x(s))\|_{\cH^{\ell-1}(0,1)}ds\\
         &\leq  \int_0^t e^{-c(t-s)}C_\Phi(2^{N_{\Phi}}\|v_x(s)\|_{\cH^{\ell-1}(0,1)}^{N_{\Phi}}+2^{N_{\Phi}}\|w_x(s)\|_{\cH^{\ell-1}(0,1)}^{N_{\Phi}})ds\\
         &\leq  \int_0^t e^{-c(t-s)}C_\Phi(2^{N_{\Phi}}\|v(s)\|_{\cH^{\ell}(0,1)}^{N_{\Phi}}+2^{N_{\Phi}}\|w(s)\|_{\cH^{\ell}(0,1)}^{N_{\Phi}})ds.
     \end{align*}
     Therefore, we have 
     \begin{align*}
           \sup_{t\in[0,T]}\bigg\|\int_0^te^{A(t-s)}(I-\Delta)^{-1}[\Phi(v_x+w_x)]_x(s)ds\bigg\|_{\cH^{\ell}(0,1)}  
         \leq\frac{2^{N_{\Phi}}C_\Phi(C')^{N_{\Phi}}}{K^{N_{\Phi}}c}(M')^{N_{\Phi}}+\frac{2^{N_{\Phi}}C_\Phi}{c}(M')^{N_{\Phi}}
     \end{align*}
     and 
     \begin{align*}
        \bigg(\int_0^T\bigg\|\int_0^te^{A(t-s)}(I-\Delta)^{-1}[\Phi(v_x+w_x)]_x(s)ds\bigg\|_{\cH^{\ell}(0,1)}^2dt\bigg)^{\frac{1}{2}}&\leq \frac{2^{N_{\Phi}}C_{\Phi}(C')^{N_{\Phi}}}{K^{N_{\Phi}}\sqrt{c}}(M')^{N_{\Phi}}+\frac{2^{N_{\Phi}}C_{\Phi}}{\sqrt{2c}}(M')^{N_F}.
     \end{align*}

      The third term can be done by noticing that 
      \begin{align*}
         &\int_0^te^{A(t-s)}(I-\Delta)^{-1}[(I-\Delta)G(v(s)+w(s))]ds = \int_0^t e^{A(t-s)}G(v(s)+w(s))ds.
     \end{align*}
     Therefore,
     \begin{align*}
           \sup_{t\in[0,T]}\bigg\|\int_0^te^{A(t-s)}G(v(s)+w(s))ds\bigg\|_{\cH^{\ell}(0,1)}  
         \leq \frac{2^{N_{G}}C_G(C')^{N_{G}}}{K^{N_{G}}c}(M')^{N_{G}}+\frac{2^{N_{G}}C_G}{c}(M')^{N_{G}}
     \end{align*}
     and 
     \begin{align*}
        \bigg(\int_0^T\bigg\|\int_0^te^{A(t-s)}G(v(s)+w(s))ds\bigg\|_{\cH^{\ell}(0,1)}^2dt\bigg)^{\frac{1}{2}}&\leq \frac{2^{N_{G}}C_{G}(C')^{N_{G}}}{K^{N_{G}}\sqrt{c}}(M')^{N_{G}}+\frac{2^{N_{G}}C_{G}}{\sqrt{2c}}(M')^{N_{G}}.
     \end{align*}

         Therefore, in order to have $\|\Gamma(w)\|_{Y^{\ell}_{0,T}}\leq M'$, we need 
     
$$C_{F+\Phi+G}\bigg(\frac{(2C')^{N_{\max}}}{K^{N_{\min}}c}+\frac{2^{N_{\max}}}{\sqrt{2c}}  \bigg)(M')^{N_{\min}}<M'$$
     where $C_{F+\Phi+G}:=C_{F}+C_\Phi+C_G$, $N_{\max}=\max\{N_F,N_{\Phi},N_{G}\}$, and $N_{\min}=\min\{N_F,N_{\Phi},N_{G}\}$ assuming $M'<1$.

Meanwhile, we also need $$\|\Gamma(w_1)-\Gamma(w_2)\|_{Y^{\ell}_{0,T}}<\|w_1-w_2\|_{Y^{\ell}_{0,T}}$$ provided that $w_1,w_2\in Y^{\ell}_{0,T}$. 
A calculation shows 
\begin{align*}
    &\Gamma(w_1)(t)-\Gamma(w_2)(t)\\
    &=\int_0^te^{A(t-s)}(I-\Delta)^{-1}\bigg[[-F(v+w_1)+F(v+w_2)]_x(s) \\
    &\quad\quad\quad +[\Phi(v_x+(w_1)_x)-\Phi(v_x+(w_2)_x)]_x(s)+(I-\Delta)[G(v+w_1)-G(v+w_2)](s)\bigg]ds
\end{align*}
Then, similar to the estimate of $\|\Gamma(w)\|_{Y^{\ell}_{0,T}}$, one has 
\begin{align*}
    &\bigg\|\int_0^t e^{A(t-s)}(I-\Delta)^{-1}[F(v+w_1)-F(v+w_2)]_x(s)ds\bigg\|_{\cH^{\ell}(0,1)} \\
    &\leq \int_0^t e^{-c(t-s)}C_F\|w_1(s)-w_2(s)\|_{\cH^{\ell}(0,1)}(\|(v+w_1)(s)\|_{\cH^{\ell}(0,1)}^{N_F-1}+\|(v+w_2)(s)\|_{\cH^{\ell}(0,1)}^{N_F-1})ds\\
    &\leq \int_0^t e^{-c(t-s)}C_F\|w_1(s)-w_2(s)\|_{\cH^{\ell}(0,1)}(2^{N_F}\|v(s)\|_{\cH^{\ell}(0,1)}^{N_F-1}+2^{N_F-1}\|w_1(s)\|_{\cH^{\ell}(0,1)}^{N_F-1}+2^{N_F-1}\|w_2(s)\|_{\cH^{\ell}(0,1)}^{N_F-1})ds
\end{align*}

Then, we can see that 
\begin{align*}
    &\sup_{t\in[0,T]}\bigg\|\int_0^t e^{A(t-s)}(I-\Delta)^{-1}[F(v+w_1)-F(v+w_2)]_x(s)ds\bigg\|_{\cH^{\ell}(0,1)}\\
    &\leq C_F\|w_1-w_2\|_{Y^{\ell}_{0,T}} (\frac{2^{N_F}(C')^{N_F-1}}{K^{N_F-1}\sqrt{c}}(M')^{N_F-1}+\frac{2^{N_F}}{\sqrt{2c}}(M')^{N_F-1})
\end{align*}
and 
\begin{align*}
    &\bigg(\int_0^T\bigg\|\int_0^t e^{A(t-s)}(I-\Delta)^{-1}[F(v+w_1)-F(v+w_2)]_x(s)ds\bigg\|_{\cH^{\ell}(0,1)}^2dt\bigg)^{\frac{1}{2}}\\
    &\leq C_F\|w_1-w_2\|_{Y^{\ell}_{0,T}} (\frac{2^{N_F}(C')^{N_F-1}}{K^{N_F-1}c}(M')^{N_F-1}+\frac{2^{N_F}}{c}(M')^{N_F-1}).
\end{align*}
By a similar argument, we have a similar upper bound for $\Phi$ and $G$. Therefore, we have 
\begin{align*}
    \|\Gamma(w_1)-\Gamma(w_2)\|_{Y^{\ell}_{0,T}}&\leq C_{F+\Phi+G}\bigg(\bigg(\frac{(2C')^{N_{\max}}}{C'K^{N_{\min}-1}c}+\frac{2^{N_{\max}}}{\sqrt{2c}}  \bigg)(M')^{N_{\min}-1}\bigg) \|w_1-w_2\|_{Y^{\ell}_{0,T}}.
\end{align*}
We first choose $K=1$, and then $M'$ small (and not exceeding 1) such that 
$$C_{F+\Phi+G}\bigg(\frac{(2C')^{N_{\max}}}{C'c}+\frac{2^{N_{\max}}}{\sqrt{2c}}  \bigg)(M')^{N_{\min}-1}<M'.$$
Then we can see that $\Gamma: Y^{\ell}_{0,T,M'}\to Y^{\ell}_{0,T,M'}$ and is a contraction mapping. Therefore, by Banach fixed point theorem, there exists a unique $w\in Y^{\ell}_{0,T,M'}$ such that $\Gamma(w)=w$. Moreover, it satisfies 
$$\|\Gamma(w)\|_{Y^{\ell}_{0,T}}\leq (\|\phi\|_{\cH^\ell(0,1)}^2+\|f\|_{L^2([0,T];\cH^{\ell-2}(0,1))}^2)^{\frac12}$$
given that $(\|\phi\|_{\cH^\ell(0,1)}^2+\|f\|_{L^2([0,T];\cH^{\ell-2}(0,1))}^2)^{\frac12}$ is sufficiently small.\\

For the term $w_t$, we have 
\begin{align*}
    \|w_t(s)\|_{H^{\ell}(0,1)}&\leq \|w(s)\|_{H^{\ell-1}(0,1)}+(1+C_{F+\Phi+G})\|w(s)\|_{H^{\ell}(0,1)}+\|w(s)\|_{H^{\ell-2}(0,1)}+C_{F+\Phi+G}\|v(s)\|_{H^{\ell}(0,1)}\\
    &\leq (3+C_{F+\Phi+G})\|w(s)\|_{H^{\ell}(0,1)}+C_{F+\Phi+G}\|v(s)\|_{H^{\ell}(0,1)}
\end{align*}
Thus, using Proposition \ref{existsub1} and taking $L^{\infty}([0,T])$ and $L^2([0,T])$ norms, we have 
\begin{align*}
    \|w_t\|_{Y^{\ell}_{0,T}}&\leq  (3+C_{F+\Phi+G})\|w\|_{Y^{\ell}_{0,T}}+C_{F+\Phi+G}\|v(s)\|_{Y^{\ell}_{0,T}}\\
    &\leq (3+C_{F+\Phi+G}+CC_{F+\Phi+G})(\|\phi\|_{\cH^\ell(0,1)}^2+\|f\|_{L^2([0,T];\cH^{\ell-2}(0,1))}^2)^{\frac12}.
\end{align*}
\end{proof}
\begin{remark}
    For $N_{\min}=2$, the argument does not work as it is possible that $$C_{F+\Phi+G}\bigg(\frac{(2C')^{N_{\max}}}{C'c}+\frac{2^{N_{\max}}}{\sqrt{2c}}  \bigg)>1.$$
 In this case, we may either solve the inequality carefully if $N_i>2$ for some $i\in \{1,2,3\}$, or assuming $C_{F+\Phi+G}$ is small.\\

     If $N_{\min}=1$, then we need to impose the smallness on $C_{F+\Phi+G}$ so that the simplified equation is of the form 
     $u_t+\Ga u_{x}-\Gb u_{xx}-u_{xxt}=\cdots$, and the operator generated by $A_{\Ga,\Gb}:=(I-\Delta)^{-1}(-\Ga u_x)-\Gb u+(I-\Delta)^{-1}(\Gb u)$ is dissipative and generates a $C_0$-semigroup.
\end{remark}

\begin{theorem}\label{thm67}
     Let $\ell_{F,\Phi,G}<2$ and $\ell>2$.     Under the assumption \hyperlink{(A)}{(A)}, we have 

$$\|w\|_{Y^\ell_{0,T}}\leq  C_{\ell}( \|\phi\|_{\cH^\ell(0,1)}^2+\|f\|_{L^2([0,T];\cH^\ell(0,1))}^2)^{\frac{1}{2}}$$
and
$$\|w_t\|_{Y^\ell_{0,T}}\leq  C'_{\ell}( \|\phi\|_{\cH^\ell(0,1)}^2+\|f\|_{L^2([0,T];\cH^\ell(0,1))}^2)^{\frac{1}{2}}$$
assuming $( \|\phi\|_{\cH^\ell(0,1)}^2+\|f\|_{L^2([0,T];\cH^\ell(0,1))}^2)^{\frac{1}{2}}$ small enough.
\end{theorem}
\begin{proof}
Differentiating \eqref{eqnonlinear} with respect to $x$ and arranging the terms, we obtain
\begin{align*}
    &w_{xxx}+w_{xxxt}-G'(v+w)[v_{xxx}+w_{xxx}]+\Phi'(v_x+w_x)[v_{xxx}+w_{xxx}]\\
    &=w_{xt}+w_{xx}+[F(v+w)]_{xx}-[\Phi'(v_x+w_x)]_{x}[v_{xx}+w_{xx}]-[G'(v+w)]_x\\
    &\quad \quad +[G''(v+w)]_x[v_{x}+w_x]^2+3[G'(u+v)]_x[w_{xx}+v_{xx}]
\end{align*}
Then, at time $t$
\begin{align*}
    &\frac{1}{2}\frac{d}{dt}\|w_{xxx}\|_{L^2(0,1)}^2+\frac{1}{2}\|w_{xxx}\|^2_{L^2(0,1)}-\int_0^1[G'(v+w)-\Phi'(v_x(t)+w_x(t))][w_{xxx}^2+w_{xxx}v_{xxx}]dx\\
    &\leq \frac{1}{4\Ge}(\|w_t\|_{\cH^1(0,1)}^2+\|w\|^2_{\cH^2(0,1)}+\|F'(v+w)[v_x+w_x]\|^2_{\cH^1(0,1)}+\|[\Phi'(v_x+w_x)]_x[v_{xx}+w_{xx}]\|^2_{L^2(0,1)}\\
    &\quad\quad\quad +\|[G'(v+w)]_x\|_{L^2(0,1)}^2\|[G''(v+w)]_x[v_{x}+w_x]^2\|_{L^2(0,1)}^2+3\|[G'(u+v)]_x[w_{xx}+v_{xx}]\|_{L^2(0,1)}^2 )+9\Ge\|w_{xxx}\|_{L^2(0,1)}^2
\end{align*}

Note that 
\begin{align*}
    \|[\Phi'(v_x+w_x)]_x[v_{xx}+w_{xx}]\|_{L^2}^2&\leq \|[\Phi'(v_x+w_x)]_x\|_{L^{\infty}}^2\|v_{xx}+w_{xx}\|^2_{L^2(0,1)} \\
    &\leq C \|\Phi'(v_x+w_x)\|_{\cH^1(0,1)}^2\|v+w\|_{\cH^2(0,1)}^2\\
    &\leq C C_{\Phi'}\|v+w\|_{\cH^2(0,1)}^{2N_{\Phi}}\leq CC_{\Phi'}2^{2N_{\Phi}}(\|v\|_{\cH^2(0,1)}^2+\|w\|_{\cH^2(0,1)}^2),
\end{align*}
and similar estimate can be obtained for $G'$. Also,
\begin{align*}
    \|[G''(v+w)]_x[v_{x}+w_x]^2\|_{L^{2}(0,1)}^2&\leq \|[G''(v+w)]_x\|_{L^{\infty}(0,1)}^2\|v_x+w_x\|_{L^{\infty}(0,1)}^2\|v+w\|_{\cH^1(0,1)}^2\\
    &\leq C'\|G''(v+w)\|_{\cH^1(0,1)}^2\|v+w\|_{\cH^1(0,1)}^4\\
    &\leq C'C_{G''}\|v+w\|_{\cH^2(0,1)}^{2N_{G}}.
\end{align*}
We have used the fact that $\|G''(v+w)\|_{\cH^1(0,1)}\leq \|G''(v+w)\|_{\cH^{2}(0,1)}\leq C_{G''}\|v+w\|_{\cH^2(0,1)}$.

The estimates for the term $F'$ and $G'$ are immediate. We also need to control the term $\|[G'(v+w)-\Phi'(v_x+w_x)]v_{xxx}\|_{L^2(0,1)}^2$, but this follows from Proposition \ref{existsub1} that 
$$\|[G'(v+w)-\Phi'(v_x+w_x)]v_{xxx}\|_{L^2(0,1)}^2\leq C_{G'}\|v+w\|_{\cH^1}$$

Therefore, take $\Ge$ small such that $\frac12-\sup_{x}G'(x)+\inf_{y}\Phi'(y)-11\Ge=:\Ge_0>0$, and by m 
    \begin{align*}
        &\|w_{xxx}(t)\|^2_{L^2(0,1)}-e^{-2\Ge_0t}\|\phi_{xxx}\|_{L^2(0,1)}^2\\
        &\leq C_{\Ge}'\int_0^te^{-2\Ge_0(t-s)}(\|v(s)\|_{\cH^3(0,1)}^{A}+\|w(s)\|_{\cH^2(0,1)}^B)ds
    \end{align*}
for some integers $A,B\geq 2$, 
so we have 
$$\|w_{xxx}\|_{Y^{0}_{0,T}}^2\leq C( \|\phi\|_{\cH^3(0,1)}^2+\|f\|_{L^2([0,T];\cH^1(0,1))}^2),$$
which implies the first inequality with $\ell=3$. The second inequality follows directly from the expression of $w_t$. 

For higher order of $\ell\in \N$, we can differentiate \eqref{eqnonlinear} with respect to $x$ $(\ell-2)$ times, and group the highest order terms to the right-hand-side, and the highest order term can be obtained only from $G'(u)[u^{(n)}]_x$ and $\Phi'(u)[u^{(n)}]_x$. We can then perform a similar argument as above to obtain the desired inequalities for $\ell\geq 4$.

For $\ell\in (n,n+1)$ for some $n\in \N\setminus\{1\}$, we can  apply nonlinear interpolation to obtain the desired inequalities.

\end{proof}

Similar arguments as the proof of Theorem \ref{main61}, we have the following.
\begin{theorem}\label{thm62}
Under Assumption \hyperlink{(A)}{(A)}, if $(\|w(\tau)\|_{\cH^\ell(0,1)}^2+\|v\|_{Y^{\ell}_{\tau,T}}^2)^{\frac12}$ is sufficiently small, then there is some $K>1$ such that 
    $$\|w\|_{Y^{\ell}_{\tau,T}}<K(\|w(\tau)\|_{\cH^\ell(0,1)}^2+\|v\|_{Y^{\ell}_{\tau,T}}^2)^{\frac12}$$
    and 
    $$\|w_t\|_{Y^{\ell}_{\tau,T}}<CK(\|w(\tau)\|_{\cH^\ell(0,1)}^2+\|v\|_{Y^{\ell}_{\tau,T}}^2)^{\frac12}$$
    for some $C\geq 1.$
\end{theorem}

\begin{prop}
Under  Assumption \hyperlink{(A)}{(A)}, suppose that $\|\phi\|_{\cH^\ell(0,1)}^2+\|f\|_{L^2([0,T];\cH^{\ell-2}(0,1))}^2$ and $\sup_{t>0}\|v\|_{Y^{\ell}_{t,T}}^{N_{\min}}$ are sufficiently small. Then, 
$$\|w\|_{Y^{\ell}_{\tau,T}}\leq C_{F,\Phi,G,K}\sup_{t>0}\|v\|_{Y^{\ell}_{t,T}}^{N_{\min}}$$
\end{prop}
\begin{proof}
   For $\ell\leq 2$, we write ($t\in [0,T]$)
    \begin{align*}
        w(t+\tau)&=e^{At}w(\tau)+\int_{\tau}^{t+\tau}e^{A(t+\tau-s)}(I-\Delta)^{-1}\bigg[[-F(v+w)]_x(s)+[\Phi(v_x+w_x)]_x(s)+(I-\Delta)[G(v+w)](s)\bigg]ds
    \end{align*}
    and $w_k:=w(kT)$. \\
    
We will show the estimate for $\tau\in [0,T]$ first and then do for other cases. Assuming $\|\phi\|_{\cH^\ell(0,1)}^2+\|f\|_{L^2([0,T];\cH^{\ell-2}(0,1))}^2$ is small so that $\|v\|_{Y^{\ell}_{0,T}}+\|w\|_{Y^{\ell}_{0,T}}<\frac{1}{2K}$, then particularly, we have $\|w(t)\|_{\cH^{\ell}(0,1)}+\|v(t)\|_{\cH^{\ell}(0,1)}<1$ for all $t\in [0,T]$. We now can choose $\sup_{t>0}\|v(t)\|_{Y^{\ell}_{t,T}}$ small so that $K^2\|w(T)\|^2_{\cH^{\ell}(0.1)}+(K^2+1)\sup_{s>0}\|v\|_{Y^{\ell}_{s,T}}^2<1$, which implies $\|w(s)\|_{\cH^{\ell}(0,1)}^2+\|v(s)\|^2_{\cH^{\ell}(0,1)}<1$ for $s\in [T,2T]$. Therefore,
\begin{align*}
    &\|w_2\|_{\cH^{\ell}(0,1)}\\
    &\leq e^{-cT}\|w_1\|_{\cH^{\ell}(0,1)}+\int_{T}^{2T}e^{-c(2T-s)}\bigg(\|F(v+w)(s)\|_{\cH^{\ell-1}(0,1)}+\|\Phi(v_x+w_x)(s)\|_{\cH^{\ell-1}(0,1)}\\
&\quad\quad\quad\quad\quad\quad\quad\quad\quad\quad\quad\quad\quad\quad\quad\quad\quad\quad\quad\quad+\|G(v+w)(s)\|_{\cH^{\ell}(0,1)}\bigg)ds \\
&\leq  e^{-cT}\|w_{1}\|_{\cH^{\ell}(0,1)}+\int_{T}^{2T}e^{-c(2T-s)}\bigg(\sum_{i\in\{F,\Phi,G\}}C_{i}2^{N_i}(\|v(s)\|_{\cH^{\ell}(0,1)}^{N_i}+\|w(s)\|_{\cH^{\ell}(0,1)}^{N_i})\bigg)ds\\
&\leq  e^{-cT}\|w_{1}\|_{\cH^{\ell}(0,1)}+\int_{T}^{2T}e^{-c(2T-s)}C_{F,\Phi,G}(\|v(s)\|_{\cH^{\ell}(0,1)}^{N_{\min}}+\|w(s)\|_{\cH^{\ell}(0,1)}^{N_{\min}})ds\\
&\leq  e^{-cT}\|w_{1}\|_{\cH^{\ell}(0,1)}+C_{F,\Phi,G}\|v\|_{Y^{\ell}_{0,T}}^{N_{\min}}+C_{F,\Phi,G}\|w\|_{Y^{\ell}_{0,T}}^{N_{\min}}\\
&\leq  e^{-cT}\|w_{1}\|_{\cH^{\ell}(0,1)}+C_{F,\Phi,G}\|v\|_{Y^{\ell}_{T,T}}^{N_{\min}}+C_{F,\Phi,G}'K^{2^{N_{\min}}}\|w_{1}\|_{\cH^{\ell}(0,1)}^{N_{\min}}+C_{F,\Phi,G}'K^{2N_{\min}}\|v\|_{Y^{\ell}_{T,T}}^{N_{\min}}.
\end{align*}
Applying the argument in the proof of Proposition \ref{Prop4}, we can conclude that 
$$\|w(t)\|_{\cH^{\ell}(0,1)}\leq C_{F,\Phi,G,K}\sup_{t\geq0}\|v\|_{Y^{\ell}_{t,T}}^{N_{\min}}$$
for all $t\in[0,2T]$, and the implicit constant is independent of $T$ and the starting time.\\

For $t\in [2T,3T]$, we further choose $\sup_{t\geq0}\|v\|_{Y^{\ell}_{t,T}}$ small such that $\sup_{s>0}\|v\|_{Y^{\ell}_{s,T}}^2<\frac{1}{K^2C_{F,\Phi,G,K}+K^2+1}$, then we can apply the argument in the previous paragraph to conclude that 
$$\|w(t)\|_{\cH^{\ell}(0,1)}\leq C_{F,\Phi,G,K}\sup_{t\geq0}\|v\|_{Y^{\ell}_{t,T}}^{N_{\min}}<\frac{1}{K^2+1}$$
for all $t\in [2T,3T]$. For $t\geq 3T$, we do not need to make  $\sup_{t\geq0}\|v\|_{Y^{\ell}_{t,T}}$ smaller because we now have 
$K^2\|w(kT)\|_{\cH^{\ell}(0,1)}^2+(K^2+1)\sup_{s>0}\|v\|^2_{Y^{\ell}_{s,T}}\leq (K^2C_{F,\Phi,G,K}+K^2+1)\sup_{s>0}\|v\|^2_{Y^{\ell}_{s,T}}<1$. Thus, we have 
$$\|w(t)\|_{\cH^{\ell}(0,1)}\leq C_{F,\Phi,G,K}\sup_{t\geq0}\|v\|_{Y^{\ell}_{t,T}}^{N_{\min}}$$
for all $t\geq 0$ provided that $\sup_{s>0}\|v\|_{Y^{\ell}_{s,T}}^2<\frac{1}{K^2C_{F,\Phi,G,K}+K^2+1}$.

Then we can follow the proof of Proposition \ref{Prop4} to show 
$$\|w\|_{Y^{\ell}_{\tau,T}}\leq C_{F,\Phi,G,K}\sup_{t>0}\|v\|_{Y^{\ell}_{t,T}}^{N_{\min}}$$ provided that $\sup_{t>0}\|v\|_{Y^{\ell}_{t,T}}^{N_{\min}}$ is small enough.

Moreover, if we replace $K$ by $CK$, we have $$\|w_t\|_{Y^{\ell}_{\tau,T}}\leq C'_{F,\Phi,G,K}\sup_{t>0}\|v\|_{Y^{\ell}_{t,T}}^{N_{\min}}$$ provided that $\sup_{t>0}\|v\|_{Y^{\ell}_{t,T}}^{N_{\min}}$ is small enough.

For $\ell=3$, we use again 
$$\|w_{xxx}(t)\|_{L^2(0,1)}^2\leq e^{-2\Ge_0(t-\tau)}\|w_{xxx}(\tau)\|_{L^2(0,1)}^2+\int_\tau^{t}e^{-2\Ge_0(t-s)}(\|w(s)\|_{\cH^{2}(0,1)}^A+\|v(s)\|_{\cH^3(0,1)}^B)ds $$
and conclude that 
$$\|w\|_{Y^{3}_{\tau,T}}\leq C''_{F,\Phi,G,K,\Ge_0}\sup_{t>0}\|v\|_{Y^{\ell}_{t,T}}^{N_{\min}}$$
assuming  that $\sup_{t>0}\|v\|_{Y^{\ell}_{t,T}}^{N_{\min}}$ is small enough.
\end{proof}

With these theorems and propositions above, we can conclude Theorem \ref{th6-1}.

\subsection{Periodic Solutions}
From now on, we assume that $f$ has temporal period $\theta$. Then, $v(t+\theta)-v(t)=e^{At}(v(\theta)-\phi)$ and 
\begin{align*}
    &w(t+\theta)-w(t)\\
    &=e^{At}(w(\theta))+\int_{0}^te^{A(t-s)}(I-\Delta)^{-1}\bigg[[-F(u(s+\theta))+F(u(s))]_x+[\Phi(u_x(s+\theta))-\Phi(u(s))]_x\\
    &\quad\quad\quad\quad\quad\quad\quad\quad\quad+(I-\Delta)[G(u(s+\theta)-G(u(s))]\bigg]ds,
\end{align*}
where $u(s)=v(s)+w(s)$.

We define $z(t)=u(t+\theta)-u(t)$. Then,   
\begin{align*}
    z(t)&=e^{A(t-\tau)}(z(\tau))+\int_{\tau}^te^{A(t-s)}(I-\Delta)^{-1}\bigg[[-F(u(s+\theta))+F(u(s))]_x\\
    &\quad\quad\quad\quad\quad+[\Phi(u_x(s+\theta))-\Phi(u_x(s))]_x+(I-\Delta)[G(u(s+\theta)-G(u(s))]\bigg]ds,
\end{align*}
and $z(0)=u(\theta)-\phi$.

Note that if $\widetilde{u}(x,t)=u(x,t+\theta)$, then 
\begin{align} \label{u-u-1}
  &(\widetilde{u}-u)_{xx}+(\widetilde{u}-u)_{xxt}+[\Phi'(\widetilde{u}_x)-\Phi'(u_x)]u_{xx}+\Phi(\widetilde{u}_x)(\widetilde{u}-u)_{xx}-[G'(\widetilde{u})-G'(u)]u_{xx}-G'(\widetilde{u})(\widetilde{u}-u)_{xx}\nonumber\\
  &=(\widetilde{u}-u)_t+(\widetilde{u}-u)_x+[-F(\widetilde{u})+F(u)]_x+[G(\widetilde{u})-G(u)]+[G'(\widetilde{u})-G'(u)]_xu_x+[G'(u)]_x(\widetilde{u}-u)_{x}
\end{align}
\begin{lemma}\label{6-10}
Suppose  Assumption \hyperlink{(A)}{(A)} holds. Then, if $u$ is the solution to \eqref{eqnonlinear}, $\widetilde{u}(x,t):=u(x,t+\theta)$ and $z=\widetilde{u}-u$, then 
$$\|z\|_{Y^{\ell}_{\tau,T}}\leq C_{\ell,F,\Phi,G}\|z(\tau)\|_{\cH^\ell(0,1)}$$ for all $\tau \geq t_0$
provided that $\sup_{t\geq t_0}\|u\|_{Y^{\ell}_{t,T}}$ is sufficiently small.
\end{lemma}
\begin{proof}
Let $\tau\geq t_0$. We first consider the case $\ell\leq 2$. Choosing $\sup_{t\geq t_0}\|u\|_{Y^{\ell}_{t,T}}$ is sufficiently small is small so that $\|u(s+\theta)\|_{\cH^{\ell}(0,1)},\|u(s)\|_{\cH^{\ell}(0,1)}<1$, we have
\begin{align*}
    &\|z(t)\|_{\cH^{\ell}(0,1)}\\&\leq e^{-c(t-\tau)}\|z(\tau)\|_{\cH^{\ell}(0,1)}+\int_{\tau}^te^{-c(t-s)}\sum_{i\in\{F,\Phi,G\}}\bigg[C_i\|z(s)\|_{\cH^{\ell}(0,1)}(\|u(s+\theta)\|_{\cH^{\ell}(0,1)}^{N_{i}-1}+\|u(s)\|_{\cH^{\ell}(0,1)}^{N_i-1})\bigg]ds\\
    &\leq  e^{-c(t-\tau)}\|z(\tau)\|_{\cH^{\ell}(0,1)}+\int_{\tau}^{t}2e^{-c(t-s)}\|z(s)\|_{\cH^\ell(0,1)}\bigg[C_F\sup_{t\geq t_0}\|u\|_{Y^{\ell}_{t,T}}^{N_F-1}+C_{\Phi}\sup_{t\geq t_0}\|u\|_{Y^{\ell}_{t,T}}^{N_\Phi-1}+C_G\sup_{t\geq t_0}\|u\|_{Y^{\ell}_{t,T}}^{N_G-1}\bigg]ds.
\end{align*}

Therefore, assuming that $\sup_{t\geq t_0}\|u\|_{Y^{\ell}_{t,T}}<1$, 
\begin{align*}
    \sup_{t\in [\tau,T+\tau]}\|z(t)\|_{\cH^\ell(0,1)}\leq \|z(\tau)\|_{\cH^\ell(0,1)}+\frac{2}{c}(C_{F+\Phi+G}\sup_{t\geq t_0}\|u\|_{Y^{\ell}_{t,T}}^{N_{\min}-1})\|z\|_{Y^{\ell}_{\tau,T}}
\end{align*}
and
\begin{align*}
    \|z\|_{L^2([\tau,\tau+T];\cH^\ell(0,1))}\leq \frac{1}{\sqrt{2c}} \|z(\tau)\|_{\cH^\ell(0,1)}+\frac{2}{c}(C_{F+\Phi+G}\sup_{t\geq t_0}\|u\|_{Y^{\ell}_{t,T}}^{N_{\min}-1})\|z\|_{Y^{\ell}_{\tau,T}}
\end{align*}
Thus, 
\begin{align*}
    \|z\|_{Y^{\ell}_{\tau,T}}&\leq C_c\|z(\tau)\|_{\cH^{\ell}(0,1)}+\frac{4}{c}C_{F+\Phi+G}\sup_{t\geq t_0}\|u\|_{Y^{\ell}_{t,T}}^{N_{\min}-1}\|z\|_{Y^{\ell}_{\tau,T}}.
\end{align*}
If we choose $\sup_{t\geq t_0}\|u\|_{Y^{\ell}_{t,T}}<(\frac{c}{8C_{F+\Phi+G}})^{\frac{1}{N_{\min}-1}}$, then 
$$\|z\|_{Y^{\ell}_{\tau,T}}\leq 2C_c\|z(\tau)\|_{\cH^{\ell}(0,1)}$$
and 
$$\|z_t\|_{Y^{\ell}_{\tau,T}}\leq 6C_c\|z(\tau)\|_{\cH^{\ell}(0,1)}.$$

We now consider $\ell\in \N\setminus\{1,2\}$. For simplicity, we consider $\ell=3$. Using \eqref{u-u-1}, 
\begin{align}\label{zxxxineq}
    &\frac{1}{2}\frac{d}{dt}\|z_{xxx}(t)\|_{L^2(0,1)}^2+\frac{1}{2}\|z_{xxx}(t)\|_{L^2(0,1)}^2-\int_0^1[G'(u)-\Phi'(\widetilde{u}_x)]|z_{xxx}(t)|^2dx \nonumber\\
    &\leq \frac{1}{4\Ge}(\|z_t\|_{\cH^1(0,1)}^2+\|z\|_{\cH^2(0,1)}^2+\|[\Phi'(\widetilde{u}_x)-\Phi'(u_x)]u_{xx}\|_{\cH^1(0,1)}^2+\|[G'(\widetilde{u})-G'(u)]u_{xx}\|_{\cH^1(0,1)}^2\nonumber\\
    &\quad\quad\quad+\|[\Phi(\widetilde{u}_x)]_xz_{xx}\|_{L^2(0,1)}^2+\|[G'(\widetilde{u})]_xz_{xx}\|_{L^2(0,1)}^2+\|F(\widetilde{u})-F(u)\|_{\cH^{2}}^2+\|G(\widetilde{u})-G(u)\|_{\cH^{1}}^2\nonumber\\
    &\quad\quad\quad + \|[G(\widetilde{u})-G(u)]u_x\|_{\cH^1(0,1)}^2+\|[G'(u)]_xz_x\|_{\cH^1(0,1)}^2)+10\Ge\|z_{xxx}\|_{L^2(0,1)}^2. 
\end{align}
Using the assumptions on $F, \Phi,G$ and the fact that $L^{\infty}(0,1)\subset \cH^1(0,1)$ and $\|u\|_{\cH^s(0,1)},\|\widetilde{u}\|_{\cH^s(0,1)}\leq C_{s}$ for all $s>\ell_{F,\Phi,G}$, we have 
\begin{align*}
    &\frac{1}{2}\frac{d}{dt}\|z_{xxx}(t)\|_{L^2(0,1)}^2+\frac{1}{2}\|z_{xxx}(t)\|_{L^2(0,1)}^2-\int_0^1[G'(u)-\Phi'(\widetilde{u}_x)]|z_{xxx}(t)|^2dx-10\Ge\|z_{xxx}\|_{L^2(0,1)}^2\\
    &\leq \frac{1}{4\Ge}(\|z_t\|_{\cH^1(0,1)}^2+\|z\|_{\cH^2(0,1)}^2+C^2C_{\Phi'}^2\|z\|_{\cH^2(0,1)}^2+C^2C^2_{G'}\|z\|_{\cH^2(0,1)}^2\\
&\quad\quad\quad+C^2C_{\Phi}^2\|z\|_{\cH^2(0,1)}^2+C^2C_{G}^2\|z\|_{\cH^2(0,1)}^2+C_F^2\|z\|_{\cH^{2}(0,1)}^2+C_G^2\|z\|_{\cH^{2}(0,1)}^2\\
    &\quad\quad\quad + C_G^2C^2 \|z\|^2_{\cH^2(0,1)}+C^2C_{G'}^2\|z\|_{\cH^2(0,1)}^2)\\
    &\leq C_{\Ge,F,\Phi,G}\|z(t)\|^2_{\cH^{2}(0,1)}.
\end{align*}
provided that $\Ge$ is chosen to be small enough. By a standard argument, we have 
$$\|z_{xxx}(t)\|_{L^2(0,1)}^2\leq e^{-\Ge_0(t-\tau)}\|z_{xxx}(\tau)\|_{L^2(0,1)}^2+C_{\Ge,F,\Phi,G}'\int_{\tau}^{t}e^{-\Ge_0(t-s)}\|z(s)\|_{\cH^2(0,1)}^2ds$$
and thus,
$$\|z_{xxx}\|_{Y^{0}_{\tau,T}}^2\leq C''_{\Ge,F,\Phi,G}\|z_{xxx}(\tau)\|_{L^2(0,1)}^2+ C'_{\Ge,F,\Phi,G}\|z\|_{Y^{2}_{\tau,T}}^2\leq C_{\Ge,F,\Phi,G}'''\|z(\tau)\|_{\cH^3(0,1)}^2,$$
which implies $\|z\|_{Y^3_{\tau,T}}\leq C_{\Ge,F,\Phi,G}^{(4)}\|z(\tau)\|_{\cH^3(0,1)}$.

For other values of $\ell$, we can apply nonlinear interpolation to conclude the inequality.
\end{proof}

\begin{proof}[Proof of Theorem \ref{th6-2}]\ \\
Note that we have established that 
\begin{align*}
    \|z_{k}\|_{\cH^{\ell}(0,1)}&\leq e^{-cT}\|z_{k-1}\|_{\cH^{\ell}(0,1)}+2C_{F+\Phi+G}\sup_{t>0}\|u\|_{Y^{\ell}_{t,T}}^{N_{\min}-1}\int_{(k-1)T}^{kT}e^{-c(kT-s)}\|z(s)\|_{\cH^{\ell}(0,1)}ds \\
    &\leq e^{-cT}\|z_{k-1}\|_{\cH^{\ell}(0,1)}+2C_{F+\Phi+G}c^{-1}\sup_{t>0}\|u\|_{Y^{\ell}_{t,T}}^{N_{\min}-1}\|z\|_{Y^{\ell}_{(k-1)T,T}}  \\
    &\leq e^{-cT}\|z_{k-1}\|_{\cH^{\ell}(0,1)}+C_{\ell,F,\Phi,G,c}\sup_{t>0}\|u\|_{Y^{\ell}_{t,T}}^{N_{\min}-1} \|z((k-1)T)\|_{\cH^{\ell}(0,1)}
\end{align*}
if $\ell\leq 2$. Then, following the argument of the proof of Theorem \ref{th2-2}, we can conclude that 
$$\|z(t)\|_{\cH^{\ell}(0,1)},\|z_t(t)\|_{\cH^{\ell}(0,1)}\leq C'e^{-C''t}\|u(\theta)-\phi\|_{\cH^{\ell}(0,1)}$$
provided that $\sup_{t>0}\|u\|_{Y^\ell_{t,T}}$ is sufficiently small, and they are also true for the norm of $Y^{\ell}_{\tau,T}$ with another implicit constants.

For $\ell=3$, using \eqref{zxxxineq}, we have (using $\|u(t)\|_{\cH^s(0,1)}\leq \sup_{t>0}\|u\|_{Y^{s}_{t,T}}$ instead)
\begin{align*}
    &\frac{1}{2}\frac{d}{dt}\|z_{xxx}(t)\|_{L^2(0,1)}^2+\frac{1}{2}\|z_{xxx}(t)\|_{L^2(0,1)}^2-\int_0^1[G'(u)-\Phi'(\widetilde{u}_x)]|z_{xxx}(t)|^2dx-10\Ge\|z_{xxx}\|_{L^2(0,1)}^2\\
    &\leq C_{\Ge,F,\Phi,G}(\sup_{t>0}\|u\|_{Y^3_{t,T}}^{2N_{\min}}+1)\|z(t)\|_{\cH^{2}(0,1)}.
\end{align*}
Therefore, 
$$\|z_{xxx}\|_{L^2(0,1)}^2\leq e^{-k(\Ge_0)T}\|\psi_{xxx}\|_{L^2(0,1)}^2+C_{\Ge,F,\Phi,G}'e^{-\Ge_0'(k-1)T}(\sup_{t>0}\|u\|_{Y^3_{t,T}}^{2N_{\min}}+1)\|\psi\|_{\cH^2(0,1)},$$
where $\psi=u(\cdot,\theta)-\phi$. Then, following the argument of the proof of Theorem \ref{th2-2}, we have the desired inequality for $\ell=3$ and so do for all $\ell>\ell_{F,\Phi,G}$.
\end{proof}
\begin{proof}[Proof of Theorem \ref{th6-3}]\ \\
The proof is basically the same as the  one of Theorem \ref{th2-3}. One need to verify is the convergence of $\widetilde{u}(\cdot,\theta)-u_{n+1}$, but this can be done by using Lemma \ref{6-10} to conclude that $\|\widetilde{u}-u(n\theta+\cdot)\|_{Y^{\ell}_{0,\theta}}\leq C\|\widetilde{\phi}-u(n\theta)\|_{\cH^{\ell}(0,1)}\to 0$ as $n\to \infty$ provided that 
$\sup_{t\geq 0}\|\widetilde{u}\|_{Y^{\ell}_{t,\theta}}+\sup_{t\geq 0}\|u\|_{Y^{\ell}_{t,\theta}}$ is sufficiently small. Everything else is the same as the proof of Theorem \ref{th2-3}.
\end{proof}

\begin{proof}[Proof of Theorem \ref{th6-4}]\ \\
We first establish the absorbing property for $\ell=1$. 

Since $\int_0^1F(u)u_xdx=0$ as $u(0)=0=u(1)$ by hypothesis,

\begin{align*}
    &\frac{1}{2}\frac{d}{dt}(\|u\|_{L^2(0,1)}^2+\|u_x\|_{L^2(0,1)}^2)+\|u_x\|_{L^2(0,1)}^2\\
    &= \int_0^1[\Phi(u_x)]_xudx+\int_0^1G(u)udx+\int_0^1[G(u)]_xu_xdx+\int_0^1f(t)udx\\
    &= -\int_0^1 [\Phi(u_x)-\Phi(0)]u_xdx+\int_0^1G(u)udx+\int_0^1G'(u)u_x^2dx+\int_0^1f(t)udx\\
    &= -\int_0^1 \Phi'(\xi u_x )u_x^2dx+\int_0^1[G(u)-G(0)]udx+\int_0^1G'(u)u_x^2dx+\int_0^1f(t)udx\\
    &\leq \int_0^1[G'(u)-\Phi'(\xi u_x)]u_x^2+\int_0^1 G'(\xi' u)u^2dx+\frac{C_M}{\Ge}\|f(t)\|^2_{\cH^{-1}(0,1)}+\Ge(\|u\|_{L^2(0,1)}^2+\|u_x\|_{L^2(0,1)}^2)\\
    &\leq \frac{1}{2}\|u_x\|_{L^2(0,1)}^2+K_G\|u\|_{L^2(0,1)}^2+\frac{C_M}{\Ge}\|f(t)\|^2_{\cH^{-1}(0,1)}+\Ge(\|u\|_{L^2(0,1)}^2+\|u_x\|_{L^2(0,1)}^2)
\end{align*}
for some $\xi,\xi'\in[0,1]$.

Therefore, by Poincar\'e inequality, we have 
\begin{align*}
     &\frac{1}{2}\frac{d}{dt}(\|u\|_{L^2(0,1)}^2+\|u_x\|_{L^2(0,1)}^2)+(\frac{1}{4}-\Ge)\|u_x\|_{L^2(0,1)}^2+(\frac
     {(c')^2}{4}-K_G-\Ge)\|u\|_{L^2(0,1)}^2\\
     &\leq \frac{C_M}{\Ge}\|f(t)\|^2_{\cH^{-1}(0,1)},
\end{align*}
which is enough to conclude the absorbing property following the proof of Theorem \ref{th2-4} provided that $\frac{(c')^2}{4}-K_G-\Ge>0$ and $\frac14-\Ge>0$. Therefore, we can conclude that $u$ is globally stable in $\cH^1(0,1)$.

We now show the absorbing property for $\ell=2$. Since 
\begin{align*}
u_t+u_{x}-u_{xx}-u_{xxt}&=[-[F(u)]_x+[\Phi(u_x)]_x]+(I-\Delta)G(u)+f(t), 
\end{align*}
 multiplying $-u_{xx}$ we have
\begin{align*}
   &\frac{d}{dt}(\|u_x\|_{L^2(0,1)}^2+\|u_{xx}\|_{L^2(0,1)}^2)+\|u_{xx}\|_{L^2(0,1)}^2-\|u_x\|_{L^2(0,1)}\|u_{xx}\|_{L^2(0,1)}\\
   &\leq C\|F(u)\|_{\cH^1(0,1)}\|u_{xx}\|_{L^2(0,1)}-\int_0^1[\Phi'(u_x)-G'(u)]u_{xx}^2dx+C\|G'(u)\|_{\cH^1(0,1)}\|u_x\|_{L^2(0,1)}\|u_{xx}\|_{L^2(0,1)}\\
   &\quad\quad +\frac{1}{4\Ge}\|G(u)\|_{L^2(0,1)}^2+\frac{1}{4\Ge}\|f(t)\|_{L^{2}(0,1)}^2+2\Ge\|u_{xx}\|_{L^2(0,1)}^2\\
   &\leq \frac{C'}{\Ge}\|F(u)\|_{\cH^1(0,1)}^2+ \frac{C'}{\Ge}\|G'(u)\|_{\cH^1(0,1)}^2\|u_x\|_{L^2(0,1)}+\frac{1}{4\Ge}\|G(u)\|_{L^2(0,1)}^2+\frac{1}{4\Ge}\|f(t)\|_{L^{2}(0,1)}^2+(K+4\Ge)\|u_{xx}\|_{L^2(0,1)}^2.
\end{align*}
We write $K=\sup_{x\in \R}G'(x)-\inf_{y\in\R}\Phi'(y)$.
Therefore, we have 
\begin{align*}
    &\frac{1}{2}\frac{d}{dt}(\|u_x\|_{L^2(0,1)}^2+\|u_{xx}\|_{L^2(0,1)}^2)+\frac{(c')^2}{2}\|u_x\|_{L^2(0,1)}^2+(\frac{1}{2}-K-4\Ge)\|u_{xx}\|_{L^2(0,1)}\\
    &\leq \sum_{i\in \{F,G',G\}} \frac{C'C_i^2}{\Ge}\|u\|_{\cH^1(0,1)}^2(1+\|u\|_{\cH^1}^{2N_i'})+\frac{1}{4\Ge}\|f(t)\|_{L^{2}(0,1)}^2\\
    &\leq C_{M,F,G',\Ge}  \sum_{i\in \{F,G',G\}}[e^{-2ct}\|\phi\|_{\cH^1(0,1)}^2+C^2\delta^2][1+[e^{-2N_{i}'ct}\|\phi\|_{\cH^1(0,1)}^2+C^{2N_i'}\delta^{2N_i'}]]+C_{\Ge}\delta^2.
\end{align*}
Following the proof of Theorem \ref{th2-4} as showing the absorbing property for $\cH^{\ell}(0,1)$, one has
\begin{align*}
    \|u_x(t)\|_{L^2(0,1)}^2+\|u_{xx}(t)\|_{L^2(0,1)}^2\leq e^{-c_{\Ge}t}M\|\phi\|_{\cH^2(0,1)}^2+e^{-c_{\Ge}t}(e^{-2ct}\|\phi\|_{\cH^1(0,1)}^2+C^2\delta^2)\Phi_{2}(\|\phi\|_{\cH^1(0,1)},\delta)+C'\delta^2
\end{align*}
provided that  $\min\{\frac{1-2K-8c}{4},0\}<\Ge<\frac{1-2K}{4}$, where $c$ is from the absorbing property for $\ell=1$.
\end{proof}

\section*{Acnowledgments}
Authors thank reviewer's precious comments and diligent work. \\

\noindent One of the authors, Taige Wang, is supported in recent years by Faculty Development Funds granted by College of Arts and Sciences, University of Cincinnati, and Taft Awards by Taft Research Center, University of Cincinnati. Authors would take this chance to thank these generous supports.

\end{document}